\documentclass{amsart}

\makeatletter
\@namedef{subjclassname@2020}{\textup{2020} Mathematics Subject Classification}
\makeatother

\usepackage{colonequals}

\usepackage{amsthm,amstext,amscd,amsbsy,amsxtra,latexsym}
\usepackage{amssymb}

\usepackage{hyperref}
\hypersetup{
	colorlinks   = true, 
	urlcolor     = blue, 
	linkcolor    = blue, 
	citecolor   = red 
}

\usepackage[msc-links]{amsrefs}

\usepackage[usenames,dvipsnames]{color}

\usepackage[all]{xy}

\usepackage[shortlabels]{enumitem}

\newtheorem{thm}{Theorem}
\newtheorem {prop}{Proposition}
\newtheorem {lem}{Lemma}
\newtheorem {defi}{Definition}
\newtheorem {exa}{Example}

\newtheorem {rem}{Remark}

\usepackage[capitalise]{cleveref}
\crefname{thm}{Thm.}{}
\crefname{prop}{Prop.}{}
\crefname{lem}{Lem.}{}
\crefname{cor}{Cor.}{}
\crefname{prob}{Prob.}{}
\crefname{figure}{Fig.}{}
\crefname{equation}{Eq.}{}
\crefname{exa}{Exa.}{}
\crefname{conj}{Conj.}{}
\crefname{defi}{Def.}{}
\crefname{rem}{Remark}{}
\crefname{section}{Sec.}{}

\DeclareMathOperator\wh{\mathcal S}   					
\DeclareMathOperator\lwh{\mathfrak s}   					

\DeclareMathOperator{\lcm}{lcm}

\DeclareMathOperator\Pic{Pic}

\DeclareMathOperator\Supp{Supp}
\DeclareMathOperator\Cl{Cl}
\DeclareMathOperator\CaDiv{CaDiv}
\DeclareMathOperator\WeDiv{WeDiv}
\DeclareMathOperator\ord{ord}

\DeclareMathOperator\sing{Sing}

\DeclareMathOperator\gl{GL}

\DeclareMathOperator\Proj{Proj}
\DeclareMathOperator\Spec{Spec}
\DeclareMathOperator\wt{wt}

\newcommand\A{\mathbb A}
\newcommand\N{\mathbb N}
\newcommand\Z{\mathbb Z}
\newcommand\Q{\mathbb Q}
\newcommand\R{\mathbb R}
\newcommand\C{\mathbb C}
\newcommand\V{\mathbb V}

\newcommand\X{\mathcal X}

\newcommand\cC{\mathcal C}
\newcommand\M{{\mathcal M}}

\newcommand\cI{\mathcal I}

\newcommand\Wc{\mathcal W}
\newcommand\cO{\mathcal O}
\newcommand\Y{\mathcal Y}
\newcommand\cL{\mathcal L}

\newcommand\w{\mathfrak q}
\newcommand\iso{\cong}
\newcommand{\kk}{{\bar k}}

\newcommand\hn{\mathfrak s}					
\newcommand\wP{\mathbb{WP}}                          	

\newcommand\x{\mathbf x}
\newcommand\y{\mathbf y}

\newcommand\il{\zeta}
\newcommand\tx{\tilde x}

\def\P{\mathbb P}
\def\a{\alpha}

\def\p{\mathfrak p}

\def\u{\mathfrak u}

\def\div{\mbox{div}}

\begin{document}

\title[LOCAL AND GLOBAL HEIGHTS ON WEIGHTED PROJECTIVE VARIETIES]{LOCAL AND GLOBAL HEIGHTS ON WEIGHTED PROJECTIVE VARIETIES}

\subjclass[2020]{Primary 	11G50; Secondary 14G40.}

\keywords{Weighted Varieties, Local and Global Weighted  Heights, Closed Subschemes}

\author{Sajad Salami}
\address{Institute of Mathematics and Statistics \& State University of Rio de Janeiro \& Rio de Janeiro, Brazil}
\email{sajad.salami@ime.uerj.br}

\author{Tony  Shaska}
\address[Tony Shaska]{Department of Mathematics \& Statistics, Oakland University, Rochester Hills, MI} 
\email{tanush@umich.edu}

\date{\today}

\begin{abstract}
We investigate local and global weighted heights a-la Weil for weighted projective spaces    via Cartier  and Weil divisors and      extend the definition of weighted heights  on weighted projective spaces from \cite{b-g-sh}   to   weighted  varieties and closed subvarieties. We prove that any line bundle   on a weighted variety    admits a locally bounded weighted  $M$-metric. Using this fact, we  define  local and global  weighted heights for  weighted varieties in  weighted projective spaces and their closed subschemes, and show  their   fundamental  properties.  


\end{abstract}

\maketitle

\setcounter{tocdepth}{3}


\section{Introduction}
Let $\w=(q_0, \cdots, q_n)$ be a tuple of weights  and $\P_{\w, k}^n $ the weighted projective space over a field $k$.  In \cite{b-g-sh} was introduced a new height on   $\P_{\w, k}^n $, called weighted height,  and proved that such height satisfies basic properties of projective heights.
This definition of weighted heights was motivated not only by its computational advantages, but also because such heights are more natural since they are defined on $\P_{\w, k}^n $ and not on some projective space $\P_k^n$ via the Veronese embedding.  
Such heights  have been used in several computations in the moduli space of curves,  rational functions; see \cite{M2, curri, mandili-sh} and are a very useful tool in using machine learning techniques in algebraic and arithmetic geometry. 
However, no complete theory of such heights exists.  For example, weighted heights in \cite{b-g-sh} were not defined analytically via Cartier divisors, local weighted heights via line bundles, global weighted heights for closed subschemes.
To the knowledge of authors this has not been done before.   

The goal of this paper is to introduce and develop the theory of weighted heights, inspired by Weil's approach. 
We achieve this by providing all the necessary tools for understanding and introducing  weighted heights, which have not been extensively covered in the literature. To accomplish this, we focus on developing the theory of Cartier divisors on weighted projective varieties, exploring the analytic structure of weighted varieties, investigating weighted blow-ups, and introducing both local and global weighted heights an showing their fundamental properties.
In our other work \cite{vojta-23}, we state some different versions of  Vojta's conjecture for weighted varieties in  terms of weighted local and global heights,  and  give an application to the greatest common divisor problem.

This paper is organized as follows.  In \cref{sect-3}  we recall some of the basic setup   for   Weil height machinery on
projective spaces and varieties.  In \cref{LWHM} we summarize all properties of local Weil heights and in \cref{GWHM} the  properties of  global Weil heights for such varieties. Such setup will be important later in the  paper to draw an analogy between Weil heights and weighted heights. 

In \cref{sect-2} we establish notation for weighted projective varieties and define  Zariski topology, Veronese embedding, and singular  locus of weighted projective  varieties.   Moreover, we introduce weighted blow-ups and exceptional divisors on weighted projective varieties. 

In \cref{sect-4}  we develop the theory of weighted heights a-la Weil. We  introduce Cartier  divisors on weighted projective varieties and show that results carry over easily to weighted projective varieties.   Moreover, we show that any line bundle on a weighted  variety $\X$ admits a locally bounded weighted $M$-metric.   Given  $\nu \in M_k$, the \textbf{local weighted height  $\il_{\widehat{D}}(-,\nu)$ with respect to} $\widehat{D}$ on weighted variety  $\X$ is defined as
\[
\il_{\widehat{D}}(\x, \nu)=- \log \| g_D(\x)  \|_v,
\]  
for  $\x      \in \X \backslash  \Supp (D)$, 
where   $v\in M$ such that $\nu=v|_k$. Properties of local weighted heights are proved in \cref{WLWHM} as   they are similar to properties of projective heights.
The \textbf{global weighted height $\hn_{\widehat{\cL}}(\x)$ }   with respect to   $\widehat{\cL}$  is defined by 
\[
\hn_{\widehat{\cL}}(\x)\colonequals \sum_{u\in M_K}^{} \il_{\widehat{\cL_g}} (\x, u), 
\]
where   $\il_{\widehat{\cL_g}} (\x, u) = -\log \| g(\x)\|_u$, 
and its   properties  are described in \cref{WGWHM}. 
In \cref{sect-4.6}  we introduce   weighted  local and global  heights associated to closed subschemes of  weighted projective varieties.


\medskip

\noindent \textbf{Notation:}
Since our goal is to provide all the technical details of the theory of weighted heights, in analogy to that of projective heights there is a real possibility of mixing up notation between different heights. 
%
Below we give a list of  notation of Weil heights and weighted heights.  Throughout the paper, the projective space (resp. weighted projective space)      over a field $k$ is denoted by $\P_k^n$  (resp. $\P_{\w, k}^n $). 

\medskip

\begin{center}
	\begin{tabular}{l|l|l}
		Terminology in projective space &      $  \P_k^n$ &    $ \P_{\w, k}^n $   \\ 
		\hline  
		& & \\
		multiplicative height over $k$ & $H_k$ &  $\wh_k$ \\  [0.5ex]
		logarithmic height over $k$ & $h_k$ &  $\lwh_k $ \\  [0.5ex]
		absolute multiplicative height & $H$ &  $\wh$ \\  [0.5ex]
		absolute logarithmic height & $h$ &  $\lwh $ \\  [0.5ex]
		local Weil height  with respect to the divisor   $\widehat{D}$  & $\lambda_{\widehat{D}}(\x,\nu)$      & $\il_{\widehat{D}}(\x, \nu)$\\   [0.5ex]
		global Weil height  with respect to the line bundle    $\widehat{\cL}$  & $h_{\widehat{\cL}}(\x)$ &  $\hn_{\widehat{\cL}}(\x)     $ \\   [0.5ex]
		local  height associated to exceptional divisor  of $\Y$ &  $ \lambda_\Y(\x, \nu) $   &  $ \il_\Y(\x, \nu) $ \\   [0.5ex]
		global  height associated to  exceptional divisor  of $\Y$  & $h_{\Y} (\x)$& $\hn_{\Y}(\x)$ \\   [0.5ex]
		absolute logarithmic  height  on $\X$ wrt divisor $D$ & $h_{\X, D}$ &   $\lwh_{\X, D}$ \\   [0.5ex]
		absolute logarithmic  local  height on $\X$ wrt  divisor $D$ & $\lambda_{\X, D} $&   $\il_{\X, D}$ \\   [0.5ex]
		Singular locus of $\P_{\w, k}^n$ &   & $\sing (\P_{\w, k}^n  )$ \\   [0.5ex]
	\end{tabular}
\end{center}

\medskip

\noindent \textbf{Acknowledgments:}  We want to thank Min Ru for helpful discussions during the period that the last version of this paper was written.

\section{Preliminaries on Weil projective heights}\label{WHVPS} \label{sect-3}
In this section, we review   Weil  heights on varieties in usual projective spaces. One can find more details on the subjects in \cite{bombieri}.

Let $k$ be an algebraic number field of degree $m=[k:\Q]$ and $\kk$ be an algebraically closed field containing $k$. We denote by  $\cO_k$   the  ring of algebraic  integers in $k$. Let   $\X$ be a variety over $k$, i.e.   an integral separated scheme of finite type over $\Spec(k)$ and  $\cO_\X$ the ring sheaf of regular  functions on $\X$. We will use $\X$ to mean $\X(\kk)$ and   $\X(k)$  for the set of $k$-rational points on $\X$. 

Denote by $M_k$ the set of all places of $k$, i.e. the equivalent  classes of  absolute values on $k$. It is a disjoint union of $M_k^0$, the set of all non-archimedian places,  and  $M_k^\infty$,   the set of all Archimedean places of $k$.   More precisely, if $\nu \in M_k^0$, then $\nu=\nu_\p$ for some prime ideal $\p \subset \cO_k$  over a prime number  $p$  such that    $\nu_{\p}|_\Q $ is the $p$-adic absolute value.  If $\nu\in M_k^\infty$, then $\nu=\nu_\infty$ and $\nu_\infty|_\Q$ is the usual absolute value $|\cdot|_\infty$ on $\Q$.
The \textbf{local degree} $n_\nu$ at $\nu \in M_k$  is defined by $n_\nu =[k_\nu:\Q_\nu]$,  where $k_\nu$ and $\Q_\nu$ are the completions with respect to $\nu$.  For each $\nu \in M_k$,  we let $|\cdot|_\nu$ be  a representative  of the equivalence class  which is the $n_\nu$-th power of the one that extends a normalized absolute value over $\Q$.  Since $k$ is a number field, then for  every $x \in k^\ast$   we have the \textbf{product formula}  $\prod_{\nu \in M_k}  |x|_\nu = 1$.  
Given  a finite field extension $K/k$,  we denote  by $M_K$ the set of places  $v$ on $K$ such that  $v\mid_k=\nu$,   for some  $\nu \in M_k$.   Then, we have  the \textbf{degree formula} as 
\[
\sum_{\substack{v \in M_K, \  v|_k=\nu}}[K_v : k_\nu]= [K:k].
\]
%
\subsection{Heights}  \label{WHPS}   
For $x \in k^\ast$, the     \textbf{multiplicative} and  \textbf{logarithmic height}  are defined by 
\begin{equation}\label{proj-height}
	H_k(x) =  \prod_{\nu \in M_k} \max \{ 1, |x|_\nu \}     \quad     \text{ and }  \quad         h_k(x) = \log H_k(x)= \sum_{\nu \in M_k}    \log |x|_\nu.
\end{equation}
For   $\tx =(x_0, \cdots , x_n)  \in k^{n+1}$ and   $v \in M_k$, we let  \[ | \tx |_\nu := \max \{ |x_i|_\nu : 0 \leq i \leq n\}.\] 
One  extends such definitions to the projective space $ \P^n(k)$ by defining the \textbf{multiplicative} and \textbf{logarithmic  height} of $\x = [x_0: \cdots : x_n]  \in \P^n(k)$ by   
\begin{equation}
	H_k(\x)  = \prod_{\nu\in M_k} \max_{0 \leq i \leq n}^{} \{   |x_i|_\nu  \},   \   \ \text{and} \ \
	h_k(\x)  =  \log H_k(\x)= \sum_{\nu \in M_k} \max_{0 \leq i \leq n}^{}  \{ \log  |x_i|_\nu  \}. 
\end{equation}
They are independent of the choice of the coordinates and therefore well defined. 

For any finite extension $K$ of $k$ and  $v \in M_K$, we normalize the absolute value $|\cdot|_v$ such that its restriction $|\cdot|_\nu$ on $k$ satisfies  $|\cdot|_\nu= |\cdot|_v^{[K_\nu: k_\nu]}$.
Using the degree formula,  for   $x \in k^\ast$   we have 
\begin{equation}
	H_k(x)  = H_K(x)^{1/[K:k]}, \ \ \text{and} \ \
	h_k(x)  =\frac{1}{[K:k]} h_K(x),  
\end{equation}
and hence  for all  $\x \in \P^n(k)$,
\begin{equation}
	H_k(\x)  = H_K(\x)^{1/[K:k]}, \ \ \text{and} \ \    h_k(\x)  =\frac{1}{[K:k]} h_K(\x).
\end{equation}
The \textbf{field of definition} of $\x  \in \P^n(\kk)$ is   $k (\x) : = k   \left(  \frac  {x_0} {x_i}, \dots , \frac {x_n} {x_i}   \right)$,  for any $i$ such that $x_i \neq 0$. 
The  \textbf{absolute multiplicative} and \textbf{logarithmic global Weil heights}  of $x\in \kk^\ast$ are defined     by
\[     H(x)=H_K(x)^{1/[K:k]}\ \ \text{and} \  \  h(x)=\frac{1}{[K:k]}h_K(\x), \]
and  for  $\x \in \P^n(\kk)$  by
\begin{equation}\label{gh}
	H(\x)=H_K(\x)^{1/[K:k]},    \ \ \text{and} \ \  h(\x)=\frac{1}{[K:k]} h_K(\x),
\end{equation}
where     $K$ is a number field containing $k(\x)$. The   absolute height is independent of the choice of $K$.   We call  $h(\x)$  the    \textbf{global Weil height} on $\P^n(\kk)$.

\subsection{$M$-bounded sets, functions,  and $M$-metrized line bundles} \label{MSFMLB}
Let $M=M_{\kk}$ be the set of places  on $\kk$  extending those  of  $ M_k$, i.e.,   if $v\in M$ then  $\nu=v|_k$  the restriction of $v$ over $k$ belongs to $ M_k$.

A function $\gamma: M_k \rightarrow \R$ is called  \textbf{$M_k$-constant} if $\gamma(\nu)=0$ for all but finitely many  $\nu \in M_k$.  We extend each  $M_k$-constant $\gamma$ to  a function $ \gamma: M \rightarrow \R$   by setting  $\gamma(v)= \gamma(v|_k)$. Given   any variety $\X$, by an \textbf{$M_k$-function} on $\X$  we mean    a map $\lambda: \X \times M  \rightarrow \R$  such that  $\lambda(\x, v)$ is $M_k$-constant or  $\lambda(\x, v)=\infty$  for all $ \x \in \X$ and $v \in M$.  Two  $M_k$-functions $\lambda_1$ and $\lambda_2$  on $\X$  are called equivalent,   and denoted by $\lambda_1 \sim \lambda_2$,  if
there is an $M_k$-constant function $\gamma$ such that 
\[
| \lambda_1(\x, v)-\lambda_2(\x, v) |     \leq \gamma(v) \;  \text{for all } \;  (\x, v)\in \X \times M.
\]
We say that an $M_k$-function $\lambda$ is   \textbf{$M_k$-bounded}  if $\lambda\sim 0$.

For an affine variety $\X$, a set  $E\subset \X \times M$ is called   an \textbf{affine $M_k$-bounded set} if there  are coordinate function $x_1,\cdots, x_n$ on $\X$ and an  $M_k$-bounded constant function  $\gamma$ such that 
\[
|x_i(\x)|_v  \leq  e^{\gamma(v)} \it  \text{ for all }    0 \leq i \leq n,\ \text{and}  \  (\x, v)\in E.
\]
The set $E$ is bounded  by a finite set of absolute values and it  is integral  with respect to the rest of absolute values. This definition is independent of choice of the coordinates $x_i$ on $\X$.   By definition,  any finite union of affine $M$-bounded sets is again an affine $M$-bounded.

For an arbitrary variety $\X$, we say that  $E\subset \X \times M$ is a \textbf{$M_k$-bounded set} if there exists a  finite cover  $\{U_i\}$ of affine open subsets of $\X$ and $M_k$-bounded sets $E_i \subset U_i\times M$   such that $E=\bigcup E_i$.

A function $\lambda: \X \times M \rightarrow \R$ is called \textbf{locally $M_k$-bounded above} if for every $M_k$ bounded subset $E \subset \X \times M,$ there exists an $M_k$-constant  $\gamma$ such that $\lambda(\x, v ) \leq \gamma(v)$ holds for $(\x, v) \in E$.   The  \textbf{locally $M_k$-bounded below}  and \textbf{locally $M_k$-bounded } functions are  defined similarly.

Recall that a line bundle $\cL$ on a variety $\X$ defined over $k$, is a covering map $\pi : \cL \to \X$ such that for each  $\x \in \X$, the  fiber  $\cL_\x \colonequals\pi^{-1} (\x)$   is a 1-dimensional vector space over $k$. 
An \textbf{$M$-metric} on  a line bundle $\cL$ is a norm $\|\cdot \|    =\left( \|\cdot \|_v \right) $   
such that for  each $v \in M$,
and each fiber $\cL_\x$   assigns a   function 
\[
\|\cdot \|_v: \cL_\x \rightarrow \R_{\geq 0},
\]
which is  not identically   zero and  satisfies:
\begin{enumerate}[\upshape(i)]
	\item $\|   \lambda \cdot  \xi  \|_v = | \lambda |_v \cdot \|  \xi \|_v $     
	for $\lambda \in \kk$ and $\xi \in \cL_\x$.
	\item If $ v_1$, $v_2 \in M$ agree on 
	$k(\x)$, then $\|\cdot \|_{v_1}=\|\cdot \|_{v_2}$ on $\cL_\x (k(\x))$.
\end{enumerate}
%
%
An $M$-metric $\|\cdot \|=\left( \|\cdot \|_v \right)$ on $\cL$ is called  \textbf{locally $M$-bounded}  if for any regular function $g\in \cO_{\X}(U)$ on an open  set $U \subseteq \X$, the function  $(\x, v) \mapsto \log \|   g(\x) \|_{v}$  on $U\times M$ is locally $M_k$-bounded.

We say that   $\cL$ is  an \textbf{$M$-metrized line bundle} on $\X$ if $\cL$ is   equipped with an  $M$-metric.   The following result  shows that  there exist an $M$-metric on any line bundle on  a variety in projective spaces; see \cite[Prop.~ 2.7.5]{bombieri}.
\begin{lem}\label{metric}
	Any line bundle $\cL$ on a variety $\X \subseteq \P_\kk^n$ defined over $k$ admits a locally bounded 
	$M$-metric $ \| \cdot  \|$.
\end{lem}

Denote  by $\widehat{\cL}$  the pair  $(\cL , \| \cdot  \|)$.   Given two pairs  $\widehat{\cL_1}= (\cL_1, \| \cdot  \|_1)$ and  $\widehat{\cL_2}= (\cL_2, \| \cdot  \|_2)$,   we define $\widehat{\cL}_1  \otimes\widehat{\cL}_2 \colonequals \left(\cL_1\otimes \cL_2, \| \cdot  \| \right)$, where 
\[
\|f\otimes g  \| = \| f  \|_1 \cdot  \| g  \|_2, \; \text{for} \  f\in \cL_{1, \x},  \  g\in \cL_{ 2, \x}, \  \text{ and  } \x \in \X.
\]
We say that $\widehat{\cL_1 }$ and $\widehat{\cL_2}$ are \textbf{isometric} if there is an isomorphism between  $\cL_1$ and $\cL_2$ which is fiber-wise an isometry.

Let $\widehat{\Pic(\X)}$ denote the  group of the isometric classes of  pairs  $\widehat{\cL}= (\cL , \| \cdot  \|)$ where  $\cL \in \Pic(\X)$. Then, the identity element of $\widehat{\Pic(\X)}$ is  $\cO_\X $ with trivial metric  
$\| 1  \|_v=|1  |_v$
and 
$  \widehat{\cL}^{-1}=(\cL^{-1}, 1/\| \cdot  \|) $
is the inverse of $\widehat{\cL} \in \widehat{\Pic(\X)}$.
Given any morphisms  $\phi: \X' \rightarrow \X$ of varieties over $k$, and $\widehat{\cL}= (\cL , \| \cdot  \|) \in  \widehat{\Pic(\X)}$,   the \textbf{pull-back} of $\widehat{\cL}$ by $\phi$ is defined as     
$\widehat{\phi^*(\cL)}  :=  \left(  \phi^*(\cL), \left(\| \cdot  \|'_v \right)    \right)$, 
such that  for   $ \x \in \X'$, any open subset $U$ of $\X$ containing $\phi(\x)$,  and  $g \in \cO_\X(U)$ we have
\[
\|  \phi^*(g) (\x) \|'_v= \| g(\phi(\x))\|_v.
\]
The pull-back induces a group homomorphism  between  $\widehat{\Pic(\X)}$ and $\widehat{\Pic(\X')}$. Under this homomorphism, any  locally bounded $M$-metrized   line bundles remain  locally bounded.
%
\subsection{Local Weil  heights} \label{LWHVPS}
We assume that the reader is familiar with  Cartier  divisors for varieties  in  projective spaces.
Given any effective  Cartier divisor $D=\{ (U_i, f_i)\}$  on $\X$,   let  $\cL_D=\cO_\X(D)$ be  the line bundle of regular functions on $D$.  It can be constructed by gluing 
$\cO_\X(D)|_{U_i}= f_i^{-1} \cO_\X( U_i)   $
and the constant section $1$  becomes a canonical invertible regular section on $\cL_D$, which we denote it by $g_D$.   We equip  $\cL_D$  with a locally bounded $M$-metric $\| \cdot \|$,
which  is possible by  \cref{metric}, 
and   denote  it by 
$ \widehat{D}=\left( \cL_D, \| \cdot \|\right)$.
Given  $\nu \in M_k$,  the \textbf{local Weil height  $\lambda_{\widehat{D}}(\cdot,\nu)$} with respect to  $\widehat{D}$ on $\X$  is 
defined to be
\begin{equation}
	\label{lwh1}
	\lambda_{\widehat{D}}(\x, \nu)=- \log \| g_D(\x)\|_v, \ \text{for}\ \x\in \X \backslash  \Supp(D),
\end{equation} 
where   $v\in M$ such that $v|_k=\nu.$

The following lemma provides a summary of all properties of local heights, which can be found in \cite[Prop.~ 2.7.10 and 2.7.11]{bombieri}  or \cite[Chap.~10]{Lang1983b}.

\begin{lem}[Local Weil heights]  \label{LWHM}   %
	For each of $\nu \in M_k$,   let $v\in M$ such that $\nu=v|_k$. Let     $\X \subseteq\P_k^n$ be a variety  defined over $k$,   and    ${\widehat{D}}, {\widehat{D}}_1,  {\widehat{D}}_2\in \widehat{\Pic(\X)}$.
	Then,  we have:
	\begin{enumerate}[\upshape(i), leftmargin=0.8cm]
		\item  
		For    $\x \not \in \Supp(D_1) \cup  \Supp(D_2) $, we have 
		\[ \lambda_{\widehat{D_1+D_2}}(\x,\nu)=  \lambda_{\widehat{D}_1 }(\x,\nu) + \lambda_{\widehat{D}_2}(\x,\nu).\]
		
		\item  
		If $\phi: \X' \rightarrow \X$ is a morphism over $k$ such that $\phi(\X')\cap \Supp(D)=\emptyset$, then 
		\[
		\lambda_{\phi^*(\widehat{D})}(\x', \nu)=\lambda_{\widehat{D}}(\phi(\x'),\nu ),  \quad \text{ for    } \; \x' \in  \X' \backslash  \phi^{-1}(\Supp(D)).
		\]
		\item    
		If $D$ is effective and  $\X$ is $M_k$-bounded (e.g $\X$ is projective),  then there exists an $M_k$-constant function $\gamma$ such that $\lambda_{\widehat{D}}(\x, \nu) \geq \gamma(\nu)$,  for  $ \x \in  \X \backslash  \Supp(D)$.

		\item  If $D=\div(f) $ for some nonzero rational function on $\X$, then  
		\[
		\lambda_{\widehat{D}}(\x, \nu)=- \log \frac{|f(\x)|_v}{|\x|_v} , \ \text{for} \ \x\in  \X \backslash  \Supp(D), 
		\]
		by giving the trivial metric   $\|1 \|_v=|1|_v$ on $\cO_\X(D)\iso \cO_\X$. 
		
		\item If $\X$ is $M_k$-bounded, $\|\cdot \|' $ is another $M_k$-bounded metric  on $\cO_\X(D)$ and  $\lambda'_{\widehat{D}}$   is the resulting local Weil height, then 
		$\lambda_{\widehat{D}}  =\lambda'_{\widehat{D}} + O(1). $
		\item If $K|k$ is a finite field extension and $u\in M_K$ over some $\nu \in M_k$, then
		\[
		\lambda_{\widehat{D}}(\x,\nu) =\frac{1}{[K:k]} \lambda_{\widehat{D}}(\x, u), \ \text{for} \ \x \in  \X \backslash  \Supp(D).
		\]
		\item 
		There are    $m, n \in \Z^{\geq 0} $, and  nonzero rational functions  $f_{i,j} $ on $\X$ for $i=0, \cdots, n_1$,     $j=0, \cdots, n_2$ such that
		\[
		\lambda_{\widehat{D}}(\x,\nu) =\max_{0 \leq i\leq n_1}^{} \min_{0\leq j\leq n_2}^{}  \log  \left| f_{ij} (\x) \right|_\nu.
		\]
	\end{enumerate}
\end{lem}

\subsection{Global  Weil heights}   \label{GWHVPS}
Let $\X\subset\P_\kk^n$ be a variety  defined over $k$ and  $\cL$  any line bundle  on  $\X$. 
Consider the pair  $\widehat{\cL}= \left( \cL,   (\|\cdot \|_{v}  )\right)  \in \widehat{\Pic(\X)}$, 
a given $\x \in \X$,  and  $K$  a finite extension of $k$ containing $k(\x)$.   For each $u\in M_K$, we choose a place  $v\in M$  over $u$  and define  
\[
\|\cdot \|_u\colonequals \|\cdot \|_{v}^{1/[K:k]}
\]
on $\cL_{\x}(k(\x))$.
By the second  condition of a   $M$-metric, one can see that it is independent of the choice of $v\in M$. We let $g$ be an invertible rational function of $\cL$  with $\x \not \in \Supp(D_g)$ where $D_g=\div(g)$.  Note that such function exists because there is an open dense trivialization in a neighborhood of $\x$. Then,   $\cO_\X(D_g)$ is a  locally $M_K$-bounded  with respect to $M_K$-metric given above.   We denote  by 
${\widehat{\cL_g}} :=\left( \cO_\X(D_g), ( \|\cdot \|_{u} ) \right)$.
The \textbf{global Weil height} $h_{\widehat{\cL}}(\x)$ of $\x\in \X$ with respect to   $\widehat{\cL}$  is defined by
\begin{equation}\label{gwh}
	h_{\widehat{\cL}}(\x)\colonequals\sum_{u\in M_K}^{} \lambda_{\widehat{\cL_g}} (\x, u),
\end{equation}
where we have   $\lambda_{\widehat{\cL_g}} (\x, u) = -\log \| g(\x)\|_u$,  assuming $v|_k=u$.     These definitions are independent of the choice of    $K$ and   $g$.  For  the following see \cite[Prop.~ 2.7.18]{bombieri}.

\begin{lem}[Global Weil height machinery]\label{GWHM}
	Let $\X$ be a variety and   $\widehat{\cL}, \widehat{\cL}_1, $ and $\widehat{\cL}_2 \in \widehat{\Pic(\X)}$.
	Then:
	\begin{enumerate}[\upshape(i), leftmargin=0.8cm]
		\item  
		$h_{\widehat{\cL}} $ depends  only on the isometry class of  $\widehat{\cL}$, i.e, if $\widehat{\cL}_1 $and $ \widehat{\cL}_2 $ are  isometric pairs, then  $h_{\widehat{\cL}_1} =h_{\widehat{\cL}_2}.$
		\item 
		If $\X$ is a complete variety or generally $M$-bounded, then  $h_{\widehat{\cL}}$ does not depends on the choice of  the locally bounded $M$-metrics up to a locally $M$-bounded constant function.
		
		\item 
		For   any $\x \in \X$, we have 
		$h_{\widehat{\cL}_1 \otimes \widehat{\cL}_2  }(\x)=h_{\widehat{\cL}_1} (\x)+h_{\widehat{\cL}_2}(\x).$
		\item  
		If $\phi: \X' \rightarrow \X$ is a morphism   over $k$, then $h_{\phi^*(\widehat{\cL})}(\x)=h_{\widehat{\cL}}(\phi(\x))$, for $ \x \in \X'$.
		\item If $\X=\P_\kk^n$ and $\cL=\cO_\X(1)$, then    $h(\x) =h_{\widehat{\cL}}(\x) + O(1).$
	\end{enumerate}
\end{lem}

\section{Weighted projective varieties}\label{sect-2}
Let $k$ be a field   and   for any integer $n\geq 1$ denote by  $\A_k^n$  (resp. $\P_k^n$)  the affine  (resp. projective)  space  over $k$. 
When $k$ is an algebraically closed field, we will drop the subscript.
For any integer $\ell \geq 1$, let   $\mu_{\ell}$  denote the group of $\ell$-th roots of unity generated by $\xi_m$, which is assumed to be contained in $k$. 

A fixed tuple of positive integers   $\w=(q_0, \dots , q_n)$ is  called  \textbf{weights}.    
Let $\V_k^n \colonequals \A_k^n \setminus \{(0, \cdots, 0)\}$ and consider the action of   $k^\ast = k \setminus \{0\}$ on $\V_k^{n+1}$ given by
\begin{equation}\label{equivalence}
	\lambda \star (x_0, \dots , x_n) = \left( \lambda^{q_0} x_0, \dots , \lambda^{q_n} x_n   \right), \      \text{for }\ \lambda \in k^\ast.
\end{equation}
Define the \textbf{weighted projective space}     $\P^n_{\w, k}$  to be  the quotient space  $\V_k^{n+1}/k^\ast$ of this action, which is  a geometric quotient since $k^\ast$ is a reductive group.   An element    $\x \in \P_{\w, k}^n $ is denoted by $\x = [ x_0 :   \dots : x_n]$  and its $i$-th coordinate  by $x_i(\x)$.  For each $i=0, \ldots, n$, we define  affine pieces  of $\P_{\w, k}^n $  by 
\[U_i=\{ \x \in \P_{\w, k}^n  :  \ x_i(\x) \neq 0\}.  \]
Hence,   $\displaystyle \P_{\w, k}^n =\cup_{i=0}^n U_i$.
We assume that the field $k$ contains a $q_i$-th root of unity $\xi_{q_i}$  for every $i=0, \cdots, n$. Then, for each $i=0, \ldots, n$,  the affine piece  $U_i$  is isomorphic to $\V_k^n/\mu_{q_i}$, the quotient  space of the action of      $\mu_{q_i}$    on $\V_k^n$ with  coordinates $z_0, \cdots, \hat  z_i, \cdots, z_n$,  given by
\begin{equation}
	\xi_i \cdot (z_0, \cdots,\hat  z_i, \cdots, z_n)  \mapsto  (\xi_i^{q_0}  z_0, \cdots, \hat  z_i, \cdots, \xi_i^{q_n}  z_n).
\end{equation}
Here,  for all $0\leq j\neq i\leq n$,  we have   $z_j =\frac{x_j}{x_i^{q_j/q_i}}$, which is similar to the case of usual projective space  $\P_k^n$.  

Weighted projective space can also be defined as a finite quotient of usual projective space.   For weights  $\w=(q_0, \dots , q_n)$, we let    
$G_\w\colonequals \mu_{q_0} \times \cdots \times \mu_{q_n}$,  which is a finite group of order $|G_\w|=q$   with $q\colonequals \prod_{i=0}^{n} q_i $.  Then, there is an action of $G_\w$ on $\P_k^n$  given by 
\begin{equation}\label{equivalence2}
	(\xi_0,  \cdots, \xi_n) \bullet [x_0: \dots : x_n]   = \left[ \xi_0  x_0: \dots : \xi_n  x_n   \right].
\end{equation}
Note that $G_\w \iso \mu_m$ if and only if  $m = \lcm (q_0, \dots , q_n)$, that is,  all of $q_i$'s are pairwise coprime. In this case, action of $G_m$    on $\P_k^n$ can be expressed as  
\begin{equation}\label{equivalence3}
	\xi^\alpha \cdot  [x_0: \dots : x_n]   =
	\left[ \xi^{\alpha /q_0} x_0: \dots : \xi^{\alpha /q_n} x_n   \right],
\end{equation}
for $0 \leq \alpha \leq m-1$, where $\xi \in G_m$ is a  $m$-th root of unity.
The morphism $\pi_0: \V_k^{n+1} \longrightarrow \V_k^{n+1}$ given by    
\[
(x_0, \cdots, x_n) \mapsto \left( x_0^{q_0}, \dots , x_n^{q_n}   \right)
\]
induces  the following diagram
\begin{equation}  \label{diag1}
	\xymatrixcolsep{2pc}
	\xymatrix{ 
		\V_k^{n+1} \ar[rr]^{\pi_0} \ar[d]^{p_\w}  &  & \V_k^{n+1}  \ar[d]^{p_\w} \\
		\P_k^n \ar[rr]^{\pi_{\w}}   \ar[dr]^{p_{\w}}  &   &  \P_{\w, k}^n \\
		& \P_k^n/ G_\w  \ar[ur]_{\iso } & }
\end{equation}
where   $p_\w$ is the canonical quotient map and  $\pi_\w:  \P_k^n \longrightarrow   \P_{\w, k}^n$ is given by
\[
[x_0: \cdots: x_n] \mapsto \left[ x_0^{q_0}: \dots : x_n^{q_n}   \right].
\]
The morphism $\pi_{\w} $ is surjective, finite, and  its   fibers  are orbits of the action of $G_\w$ on $\P_k^n$, see 
\cite[Chap. V, Props. 1.3 and  1.8]{Raynaud1971}.


$\P_{\w, k}^n (k)$  will denote the set of $k$-rational points of $\P_{\w, k}^n$.   When $k$ is  algebraically closed and there is no room for confusion sometimes    $\P_\w^n$ is used instead of  $\P_{\w, k}^n$.

\subsection{Zariski topology on weighted projective spaces}

Consider the ring of polynomials $k[x_0, \ldots , x_n]$ and assign to every variable $x_i$ the weight  $\wt (x_i)=q_i$,
for all $i=0, \ldots , n$.   Every polynomial is a sum of monomials $x^d= \prod x_i^{d_i}$ with    $ \wt (x^d) = \sum d_i q_i$. 

Let $f  \in k[x_0, \dots, x_n]$,  where $\wt (x_i) = q_i$, for $i=0, \dots, n$.  Then,  $f$   is called a  \textbf{weighted homogeneous\footnote{{In some papers on weighted projective spaces,  a weighted homogeneous polynomial is also called \textbf{quasihomogeneous} polynomial.}} polynomial of degree $d$} if each monomial in $f$ is weighted of degree $d$, i.e. 
\[ f (x_0, \dots , x_n) = \sum_{i=1}^t a_i \prod_{j=0}^n x_j^{d_j}, \; \text{ for } \;  a_i \in k  \, \text{ and } \,  t \in \N \]
and for all $ 0 \leq i \leq n$, we have that $\sum_{i=1}^n q_i d_j = d$.  
For  every $\lambda \in k^\ast$ and any  weighted homogeneous polynomial $f$,   we have
\[  f(\lambda^{q_0} x_0, \lambda^{q_1} x_1, \dots , \lambda^{q_n} x_n) = \lambda^d f(x_0, \dots , x_n),\]
We denote by   $k_\w [x_0, \dots, x_n]$  the \textbf{set of weighted homogeneous polynomials} over $k$. 
%
It is a subring of $k[x_0, \dots, x_n] $  and therefore  a Noetherian ring.  By $k_\w[x_0, \dots, x_n]_d$ we mean the \textbf{additive group of all weighted homogeneous polynomials of degree $d$}.

Let  $\a=[\a_0 :  \cdots : \a_n] \in \P^n_{\w, k}$ and $ f \in k_\w[x_0, \dots, x_n]_d$.     Then, for any  $\lambda \in k^\ast$, we have   $\a =   [\lambda^{q_0} \a_0:  \cdots : \lambda^{q_n} \a_n]$.   Since
\[
f   \left(  \lambda^{q_0} \a_0, \dots, \lambda^{q_n} \a_n    \right) = \lambda^d \, f( \a_0, \dots, \a_n)  = 0,
\]
then $\a$ being a zero of $f$ is well-defined for all $\a\in \P_{\w, k}^n $. 

A \textbf{weighted hyperplane} in $\P_{\w, k}$  is a weighted homogeneous polynomial of degree $m$.  
Hence, it is the set of points $\x=[x_0: \ldots : x_n] \in \P_{\w, k}$ satisfying a polynomial of the form
\begin{equation}
	\label{whyper}
	\ell(\x)=a_0 x_0^{m/q_0}+ a_1 x_1^{m/q_1} + \cdots+ a_n x_n^{m/q_n} = \sum_{i=0}^n a_i x_i^{\frac m {q_i}}
\end{equation}
Notice that if $\w=(1, \ldots , 1)$ all definitions agree with those of   $\P^n$.

An ideal $I \subset k_{\w} [x_0, \ldots, x_n]$   is called a \textbf{weighted homogeneous ideal}  if every element of $f\in I$ can be written as $f = \sum_{i=0}^{d} f_i$ where $f_i \in k_\w[x_0, \dots, x_n]_i\cap I$ with  $\deg(f_i)=i.$    The \textbf{sum} of two weighted homogeneous ideals $I$ and $J$,  is denoted by $I+J$ and is defined to be 
\[ I+J = \{ f+g | f \in I, g \in J. \} \]
If $I$ and $J$ are weighted homogeneous ideals in $k_\w[x_0, \dots, x_n]$,  then $I +J$ is also an weighted homogeneous 
ideal in $k_\w[x_0, \dots, x_n]$.
The \textbf{product} of two weighted homogeneous ideals $I$ and $J$ is denoted by $IJ$ and  is defined to be the ideal  
\[ 
IJ = \langle   \{ fg \,  | \, f \in I, g \in J \} \rangle.
\]
For any given    weighted homogeneous ideal   $I$, we define \textbf{weighted projective variety of $I$} by
\begin{equation}\label{def-wpv}
	V (I)  = \left  \{ \x \in \P_{\w, k}^n \, \left |\, \frac{}{} \right.      f(\x) = 0 \, \, \text{ for all } \, \, f \in I \right \}
\end{equation}
Let $I$ and $J$ be weighted homogeneous ideals. Then the following hold:
\begin{enumerate}[\upshape(i), nolistsep] 
	\item  $V(I)\cap V(J)=V(I+J)$
	\item    $V(I)\cup V(J)=V(IJ)$
	\item   $\P_{\w, k}^n = V(0)$
\end{enumerate} 
%
Conversely,  given any $V \subset \P_{\w, k}^n $   the \textbf{weighted homogeneous ideal associated to $V$} is given by
\[
I(V) = \left\{ f \in k_\w [x_0, \dots, x_n] \, \left | \frac{}{} \right. \, f(\x) = 0 \, \, \text{ for all } \, \,  \x \in V   \right \} 
\]
A weighted homogeneous ideal $I$ is called a \textbf{radical weighted homogeneous ideal} if $f\in k_\w[x_0, \dots, x_n]$ such that $f^r\in I$ for an integer $r\geq 1$ then $f\in I$.

\begin{lem}
	Let $V \subset \P_{\w, k}^n$ be a weighted projective variety. Then, weighted homogeneous ideal $I(V)$ associated to $V$  is a radical weighted homogeneous ideal. 
\end{lem}

\begin{proof} Let $f$ and $g$ be two polynomials in $I(V)$. Then, $f(P) = g(P) = 0$ for all points $P \in V$, i.e. they both vanish at all points $P$ in the variety $V$ then so does $f+g$ and $f h$ where $h$ is any polynomial in $I(V)$. Therefore, $I(V)$ is a weighted homogeneous ideal. 
	
	Since, $k_\w[x_0, \dots, x_n]$ is Noetherian,     then $I(V)$ is finitely generated, say  $I(V)= \langle f_1, \dots, f_n\rangle$.  However,  $f_i \in k_\w[x_0, \dots, x_n]$ for all $i$ and therefore every $f_i$ is weighted homogeneous polynomial. Hence $I(V)$ is weighted homogeneous ideal since it is generated by finitely many weighted homogeneous polynomials. 
	
	Finally let us prove that $I(V)$ is radical. Let $f^r \in I(V)$. Then, for all points $P \in V$ we have that $f^r (P) = 0$. But since $f \in k_\w [x_0, \dots, x_n]$, which is an integral domain, then $f^r (P) = \left(f(P) \right)^r = 0$ implies that $f(P) = 0$ for all $P\in V$. Therefore, $I(V)$ is radical.  This completes the proof. 
\end{proof} 

For weighted projective varieties $V $ and $ W$  then we say that $V$ is a \textbf{weighted  subvariety} of $W$ if $V \subset W$.
It can be shown that     any finite union of weighted   projective varieties is a weighted  projective variety.
Furthermore,    an arbitrary intersection of weighted projective varieties is a weighted projective variety.
A weighted projective variety  is said to be \textbf{irreducible}   if it has no non-trivial decomposition into subvarieties. 
We notice that any weighted projective varieties are projective varieties too.  Hence, we can define the  Zariski topology for weighted projective varieties.  \textbf{Zariski topology} on a weighted projective space $\P_{\w, k}^n$ is given by defining closed sets of $\P_{\w, k}^n$ to be those of the form $V(I)$ for some weighted homogeneous ideal $I \subset k_\w[x_0, \dots, x_n]$. 

\begin{defi}
	\textbf{Zariski closure} of a subset $S$  of a weighted projective space $ \P_{\w, k}^n$ is the smallest weighted projective variety that contains $S$.
\end{defi}

\begin{rem}
	Let $S\subset \P_{\w, k}^n $.   Then,  $V(I(S))$ is the Zariski closure of $S$.
	The proof is similar to the case of  projective varieties. 
\end{rem}

\begin{exa} 
	Let $\w=(q_0,  q_1, q_2)$ and   $f \in k_\w [x, y, z]_d$. Then,   $ V(f) \subset \P_{\w, k}^2 $ is a degree $d$-plane curve in $\P_{\w, k}^2$.
\end{exa}

The following  gives the third equivalent definition of weighted projective space   in language of schemes,   see  \cite[Subsection 1.2.2]{dolga} or \cite[Theorem 3A.1]{Beltrametti1986}.

\begin{prop}\label{Bprop1}
	$\P_{\w, k}^n$ is isomorphic to  $\Proj  \left(k_\w [x_0, \dots, x_n]\right)$.
\end{prop}

For the rest of this paper, by  a \textbf{weighted variety} we mean    an integral, separated subscheme of finite type in $\Proj(k_\w[x_0,\cdots, x_n])$. In other words,  $\X \subseteq \P_{\w, k}^n$ is a weighetd variety if there are 
$f_1, \cdots, f_t \in k_\w [x_0, \dots, x_n]$ such that $\X$ is isomorphic  to the $k$-scheme 
$\Proj \left( \frac{k_\w [x_0, \dots, x_n]}{\left\langle f_1, \cdots, f_t  \right\rangle }\right).$

A weighted space $\P_{\w, k}^n $ is called \textbf{reduced } if   $\gcd (q_0,  \cdots , q_n)   = 1$.   It is called  \textbf{normalized} or \textbf{well-formed} if   
\[ \gcd (q_0, \dots , \hat q_i, \dots , q_n)   = 1, \quad \text{for each } \;  i=0, \dots , n. \]

\subsection{Veronese map}
Let $R$ be a graded ring and $d \geq 1$ be an integer. Its \textbf{$d$-th truncated ring } is the subring $R^{[d]} \subseteq R$  defined by 
\[
R^{[d]} \colonequals\bigoplus_{d \vert n} R_n = \bigoplus_{i \geq 0} R_{d i}.
\]
Clearly we have  the embedding $R^{[d]}  \hookrightarrow  R$, which is called the \textbf{$d$-th Veronese embedding},
implying   that  $\Proj (R^{[d]})  \iso  \Proj (R)$ by  \cite[Prop.~ 2.4.7]{Grothendieck1961}.
Moreover, the sheaf $\cO(1)$ on $\Proj (R^{[d]})$ corresponds  via the isomorphism to   $\cO(d)$ on $\Proj (R)$.

\begin{prop} \label{reduction}
	Given  any  tuple of weights   $\w=(q_0, \dots, q_n)$, the following hold:
	\begin{enumerate}[\upshape(i)]
		\item    Any weighted projective space  $\P^n_{\w, k}$ is isomorphic to  $\P^n_{\w',k}$, where  $\w'$ is a reduced tuple of weights.
		\item   If  $\P^n_{\w, k}$  is reduced and  $d_i= \gcd (q_0, \cdots, \hat q_i, \cdots, q_n)$ for $0\leq i \leq n$, 
		then $\P^n_{\w, k} \iso  \P^n_{\w', k} $ with  $\w'=\left(\frac{q_0}{d_i}, \dots, \frac{q_{i-1}}{d_i}, q_i, \frac{q_{i+1}}{d_i}, \dots, \frac{q_n}{d_i}\right).$
		\item   Any $\P^n_{\w, k}$  is isomorphic to a reduced and  well-formed one.  
		\item If $\w$ is reduced and all of   $m/q_i$ are co-prime, where $  m= \lcm \left( q_0, \cdots ,  q_i\right),$  then $\P^n_{\w, k}$ is isomorphic to $\P_k^n$ by   $\phi_m : \P^n_{\w, k}  \longrightarrow  \P_k^n$ defined as 
		%
		\begin{equation}\label{veronese}
			\phi_m([x_0, \dots , x_n])  = [x_0^{m/q_0}, x_1^{m/q_1}, \dots , x_n^{m/q_n}].
		\end{equation}
	\end{enumerate}
\end{prop}

\proof  Let $d= \gcd (q_0, \dots , q_n)$, $R = k_\w [x_0, \dots, x_n]$, and     $R^{[d]}$ be the $d$-th truncated subring of $R$. Then, $R^{[d]}=k_\w[x_0^d, \dots, x_n^d]$ and by  \cref{Bprop1} we have 
\[
\P^n_{\w, k} = \Proj (R)\iso\Proj (R^{[d]})=\P^n_{ \w', k},\ 
\text{with}\  \w'=\left(\frac{q_0}{d}, \dots, \frac{q_n}{d} \right),
\]
under the isomorphism  
\begin{equation}\label{reduced}
	[x_0: \cdots: x_n] \rightarrow [y_0: \cdots : y_n]\colonequals[x_0^d : x_1^d: \cdots:    x_n^d].
\end{equation}
This shows that $\P_{\w, k}^n$ is isomorphic to a reduced  weighted projective space $\P^n_{\w', k}$, i.e., 
with $\w'=(q'_0, \cdots, q'_n)$ such that  $\gcd(q'_0, \cdots, q'_n)=1$.  This completes the proof of part (i).

Now, we assume that  $\gcd(q_0, \dots, q_n)=1$  and let     $d_i= \gcd (q_0, \cdots, \hat q_i, \cdots, q_n)$, for  $0\leq i \leq n$. Then,   $\gcd(d_i, q_j)=1$, for all $0 \leq i\neq j \leq n$. If $x_0^{p_0} \cdots x_n^{p_n}$ is a monomial of degree $p d_i$ for an integer $p\geq 1$,  then 
\[
p_0 q_0 +\cdots +p_n q_n = p d_i,
\]
and so $d_i $ divides $p_i q_i$,  and hence $d_i | p_i$.  This implies that $x_i$ only appears in  $ R^{[d_i]}$ as $x_i^{d_i}$.
Thus, we have   $R^{[d_i]} = k [x_0, \dots, x_{i-1}, x_i^{d_i}, x_{i+1}, \dots, x_n]$
and hence   
\begin{equation}\label{well-formed1}
	\P^n_{\w, k}  = \Proj (R)   \iso   \Proj (R^{[d_i]}) =  \P^n_{\w',k},
\end{equation}
with $\w'=\left(\frac{q_0}{d_i}, \dots, \frac{q_{i-1}}{d_i}, q_i, \frac{q_{i+1}}{d_i}, \dots, \frac{q_n}{d_i}\right)$ under the isomorphism 
\[ [x_0: \cdots: x_n] \rightarrow [y_0: \cdots : y_n]\colonequals[x_0 : \cdots : x_i^{d_i}  : \cdots:    x_n],
\]
see 
\cite[Prop. 3]{b-g-sh}  for  more details.     Thus,  the part (ii) is proved. 

One can  conclude  part (iii)  by  repeatedly   using  (ii).   Indeed,   by  defining 
\[
d_i= \gcd (q_0, \cdots, \hat q_i, \cdots, q_n), \; 
a_i  = \lcm (d_0, \cdots, \hat d_i, \cdots, d_n),  \;
a =\lcm(d_0, \cdots, d_n), 
\]
for all $0 \leq i \leq n$, one can easily check the following:
\begin{enumerate}[\upshape(1)]
	\item   $a_i \vert q_i$,  \  $\gcd (a_i, d_i)=1$  and  $ a_i d_i=a$   for  $  0 \leq i \leq n;$ 
	\item   $\gcd (d_j, d_i)=1$, and   $d_j \vert q_i$, for  $   0\leq i \neq j \leq n.$
\end{enumerate}
Then, denoting  by   $R^{[\textbf{d}]} \colonequals k_\w [x_0^{d_0}, \cdots, x_n^{d_n}]$,  we have 
\[ 
\P^n_{\w,k }  = \Proj (R)   \iso   \Proj (R^{[\textbf{d}]}) =  \P^n _{\w',k}\ \text{with}\ \w'=(q'_0, \dots , q'_n).
\]
where $q'_i=q_i/a_i$ for all $0 \leq i \leq n$, under the morphism 
\begin{equation}\label{well-formed}
	[x_0: \cdots: x_n] \rightarrow [y_0: \cdots : y_n]\colonequals [x_0^{d_0}: \cdots: x_n^{d_n}].
\end{equation}
Since     $\gcd (q'_0, \cdots, \hat {q'}_i, \cdots, q'_n)=1$ for all $0 \leq i \leq n$, then  $\P^n_{\w', k}$ is a well-formed   weighted projective space;
see     \cite[Prop.~ 2.3]{MR3285300} 
for more details.  This completes the proof of   part (iii).

If  $a_i =q_i$   for all  $0 \leq i \leq n$ in the above discussion,   then  $\P^n_{\w, k }    \iso \P_k^n$. 
This   holds if $m/q_i$ are all co-primes, where $m=\lcm (q_0, \cdots, q_n)$    
The isomorphism  is given by \cref{veronese}.
\qed

We  call  the isomorphism  $\phi_m$ given in \cref{veronese}   the \textbf{Veronese map}.

\begin{exa}[The space $\M_2$] \label{exam1}
	Consider the weighted projective moduli space  of genus $2$ curves,  say $\P_{\w, k}^3$ for $\w=(2,4,6,10)$.  
	
	Let $d_0:= \gcd (4,6,10)=2$, $d_1 = \gcd(2, 6, 10)= 2$, $d_2= \gcd(2, 4, 10)=2$, $d_3:=\gcd(2, 4, 6)=2$ and 
	$a_0 =\lcm (2,2,2)=2 = a_1=a_2=a_3$, and $a=\lcm(2,2,2,2)=2$.   The new set of weights is $q_i^\prime = \frac {q_i} {a_i}$. Hence 
	$ \w^\prime = (1, 2, 3, 5)$. 
	Thus,  the morphism   $\P_{(2,4,6,10), k}^3          \to \P_{(1,2,3,5), k}^3 $, given by 
	\begin{equation}
		[x_0: x_1:x_2:x_3]     \to [y_0: y_1 :y_2:y_3]=\left[  x_0^2 : x_1^2 : x_2^2: x_3^2        \right]
	\end{equation}
	is an isomorphism,  from \cref{well-formed}.
	%
	Then $q=2\cdot 3 \cdot 5= 30$ and the  Veronese embedding is 
	\[\left[  J_2: J_4: J_6: J_{10} \right] \longrightarrow \left[  J_2^{30}: J_4^{15}: J_6^{10}: J_{10}^6 \right].\]
	Since $J_{10}$ is the discriminant then $J_{10}    \neq 0$, then 
	\[
	\left[  J_2^{30}: J_4^{15}: J_6^{10}: J_{10}^6 \right] =  \left[  \frac {J_2^{30}}   {J_{10}^6}:    \frac { J_4^{15}}  {J_{10}^6} :  \frac {J_6^{10}} {J_{10}^6}: 1    \right]
	\] 
	Thus, two  genus curves  are isomorphic if and only if they have the same 
	$i_1:=\frac {J_2^{30}}   {J_{10}^6}$,  $i_2:=\frac { J_4^{15}}  {J_{10}^6}$,  and $i_3:=  \frac {J_6^{10}} {J_{10}^6}$ 
	invariants.  Such invariants $i_1$, $i_2$, $i_3$ are $\gl_2 (k)$-invariants and sometimes are called \emph{absolute invariants}.  To avoid invariants with such high degrees sometimes different invariants have been used,  where $i_1 = \frac {J_4} {J_2^2}$, $i_2=\frac {J_2J_4-J_6}  {J_2^3}$, and $i_3=\frac {J_{10}} {J_2^5}$, but then we have to define new invariants for the locus $J_2=0$; see \cite{M2}, and many other authors. 
\end{exa}

Example above shows the benefits of weighted projective spaces  from a computational point of view, since it is much easier to compute with $\left[  J_2: J_4: J_6: J_{10} \right] $ because the coordinates have much smaller degrees instead of $\left[  J_2^{30}: J_4^{15}: J_6^{10}: J_{10}^6 \right]$.
It was exactly this  fact and computational efforts in \cite{M2}   which  led to the definition  of the weighted general common divisors and weighted heights  in \cite{mandili-sh} and  \cite{b-g-sh}; as we will see in detail in \cref{sect-4}.
$\M_2$ is a very nice example of doing explicit computations, however GIT  guarantees that the theory works in every genus. 

\subsection{Singular locus of  weighted projective varieties}
Singularities of   $\P_{\w, k}$ are classified    in the following proposition,    see    \cite{dolga} or \cite{Beltrametti1986} for its proof.

\begin{prop}\label{Bprop2}
	$\P_{\w, k}^n$ is    an irreducible,  normal and Cohen-Macaulay  variety having  only    cyclic quotients singularities.    
	Moreover,  if  $\P_{\w, k}^n $ is non-singular, then it is isomorphic to $\P_k^n$.
\end{prop}

We let  $d=\gcd (q_0, \ldots , q_n)$  
and  denote by  $\sing(\P_{\w, k}^n )$ the \textbf{singular locus} of $\P_{\w, k}^n$.  Then, following the proof of \cite[Prop. 7]{Dic-Dim},  one can show that
\[
\sing(\P_{\w, k}^n ) = \left\lbrace  \x \in  \P_{\w, k}^n  : \ \gcd_{ i \in J(\x) } \left( q_i \right) > d \right\rbrace
\]
For $\x\in \P_{\w, k}^n$ denote by  $J(\x):= \{ j  \ :\ x_j        (\x) \neq 0\}), $
the set of indexes where $\x$ has non-zero coordinates. 
Let $m=\lcm (q_0, \cdots, q_n)$, $p$ a prime dividing $m$, and 
\[
S_\w  (p) = \left\lbrace \x \in  \P_{\w, k}^n  : \     d p\mid q_i \ \text{for all} \ i \in J(\x)  \right\rbrace.
\]
The singular locus decomposes into irreducible components as  
\[
\sing(\P_{\w, k}^n ) =\bigcup_{\text{primes} \ p  | m, \ }^{} S_\w(p),
\]
where   only the maximal sets are considered in the union. 
The proof can be easily extended from that of   \cite{Dimca1986} see remark below. 

\begin{rem}In most papers the weighted projective space is assumed well formed. This is not really a restriction since every weighted projective space is isomorphic to a well-formed space. Then 
	\begin{equation}\label{s-w-p}
		S_\w  (p) = \left\lbrace \x \in  \P_{\w, k}^n  : \    p\mid q_i \ \text{for all} \ i \in J(\x)  \right\rbrace.
	\end{equation}
	and the singular locus is 
	\[
	\sing(\P_{\w, k}^n ) = \left\lbrace  \x \in  \P_{\w, k}^n  : \ \gcd_{ i \in J(\x) } \left( q_i \right) > 1 \right\rbrace
	\]
	see \cite[Prop. 7]{Dic-Dim}.     Since $\P_{\w, k}^n$ is well-formed then    $\x \in \sing(\P_{\w, k}^n )$ implies that  $x_i(\x)=0$ for at least one index   $i \in \{0, \ldots , n\}$.
\end{rem}

\begin{exa}[$\M_2$ again] 
	Let us consider again \cref{exam1}.
	
	Consider $\P_\w^3$ for $\w=(2,4,6,10)$.   Then $m=\lcm (2,4,6,10)=60$.  The only primes  dividing  $m=60 $ are $p=2, 3, 5$.  Then 
	\[
	\begin{split}
		S_{\w} (2)   & = \{ [0: t : 0:0] \in \P_\w^3   \}, \\
		S_{\w}  (3)  & = \{ [0: 0 : t :0] \in \P_\w^3   \}, \\
		S_{\w}  (5)  & = \{ [0: 0 : 0: t] \in \P_\w^3   \}
	\end{split}
	\]
	Hence,     $\sing  \P_{\w, \Q}^3  = S_\w (2)  \cup S_\w (3)  \cup S_\w (5)$.    
	
	One can take $\w^\prime=(1,2,3,5)$  and $\P_{\w^\prime, \Q}^3 $. Then $m=\lcm (1,2,3,5)=30$.  Only primes $p=2, 3, 5$  divide $m$.
	Then,
	\[
	\begin{split}
		S_{\w\prime} (2)  &= \{ [0: t : 0:0] \in \P_\w^3   \}, \\
		S_{\w\prime}  (3)  &= \{ [0: 0 : t :0] \in \P_\w^3   \}, \\
		S_{\w\prime}  (5)  &= \{ [0: 0 : 0: t] \in \P_\w^3   \}.
	\end{split}
	\]
	Hence,     $\sing  \P_{\w^\prime, \Q}^3  = S_\w^\prime (2)  \cup S_\w^\prime (3)  \cup S_\w^\prime (5)$.          
	\qed
\end{exa}

For a fixed prime $p$ such that $p \nmid m$, then $S_\w (p)=\emptyset$.  If $p\mid  m$ then denote 
\[ J (p) = \{  j \, \vert \, \text{ such that } \,  p \mid q_j \}, \quad  \text{ and } \quad n_p=\# J(p).
\]
Then    $S_\w(p) \neq \emptyset$    is isomorphic to the  weighted projective space
$\P_{\w^\prime, k}^{n_p}$, where $\w^\prime=(q_{i_1}, \cdots, q_{i_{n_p}})$ with $ i_\ell \in J(p)$ for $1 \leq \ell \leq n_p.$ 
Moreover, as a consequence of the normality of $\P_{\w, k}^n $, we have
$ \text{Codim}_{\P_{\w, k}^n } (\sing(\P_{\w, k}^n )) \geq 2$. 
This means that $\P_{\w, k}^n $ is regular in codimension one. In particular, if $ q_i$'s  are mutually coprime   and $q_i>1$, then  
\[\sing(\P_{\w, k}^n )=\{\x_i=[0: \cdots: 1: \cdots: 0] : \  0\leq i \leq n \}.\]

Next we consider the  canonical  quotient map  $p_\w:  \V_k^{n+1}\rightarrow \P_{\w, k}^n $,  which induces the surjective morphism 
$\pi_\w: \P_k ^n  \rightarrow \P_{\w, k}^n$.
Let $\X$ be a weighted subvariety  of  $\P_{\w, k}^n $.  
The  \textbf{punctured affine cone}  over $\X$ is  $\cC^*_\X=p_\w^{-1} (\X)$.  The \textbf{affine cone}   $\cC_\X$ over $\X$ is the closure of $\cC^*_\X$ in $\A_k^{n+1} $.  
The origin point $\textbf{0}=(0, \cdots, 0 )$ refers to the vertex of  $\cC^*_\X$.
We note that $k^\ast$ acts on the  punctured affine cone $\cC^*_\X=p_\w^{-1} (\X)$  to result  $\X= \cC^*_\X / k^\ast$.  Moreover,  $\cC^*_\X$ has no isolated singularities.

A  weighted subvariety   $\X$  of  $\P_{\w, k}^n $ is called \textbf{quasi-smooth}  of dimension $m$   if its affine cone  $\cC_\X$ is smooth   variety of dimension $m+1$ outside its vertex.     The singularities of a quasi-smooth variety $\X$  are due to the $k^\ast$-action  and hence are cyclic quotients singularities.
Furthermore, by \cite[Cor. 5.9]{Beltrametti1986}, if $\X \subset \P^n_{\w, k}$ is subvariety such that
$\X \cap \sing (\P^n_{\w, k})= \emptyset$, then  $\X$ is non-singular if and only if $\X$ is quasi-smooth.

A  weighted subvariety   $\X$  of  $\P_{\w, k}^n $ of codimension $c$  is called \textbf{well-formed} if  $\P_{\w, k}^n $ itself is well-formed and  $\X$ contains no codimension $c+1$ singular stratum  of  $\P_{\w, k}^n $. Hence, any codimension $1$ stratum of  a well-formed variety $\X $ is either nonsingular on  $\P_{\w, k}^n $ or  it is equal to $\X \cap \Y$, where $\Y$ is a codimension $1$ stratum of  $\P_{\w, k}^n $. This means that $\text{Codim}_{\X }(\X \cap \P_{\w, k}^n ) \geq 2$.

Given   a weighted polynomial $f \in k_\w[x_0, \cdots, x_n]$ of degree $d$, let $\X_d$  denotes the hypersurfaces  defined by $f$. It is called a \textbf{linear cone} if  $d= q_i$ for some $0 \leq i \leq n$, i.e, it is defined by $x_i+g$ with  $g \in k$. A linear cone is well-formed if and only if it is isomorphic to  $\P_{(q_0, \cdots, \hat q_i, \cdots,  q_n), k}^{n-1}$. In the case of hypersurfaces, $\X_d$ is well-formed if and only if the following hold:

\begin{itemize}[\upshape(i)]
	\item   $\gcd(q_0, \cdots, \hat q_i, \cdots, q_n)=1$ for all  $0 \leq i  \leq n$;
	
	\item $\gcd(q_0, \cdots, \hat q_i, \cdots,   \hat q_j, \cdots, q_n) $ divides $d$ for  $0 \leq i\neq j \leq n$.
\end{itemize}

For more    on well formed   subvarieties  of  $\P_{\w, k}^n $ of codimension $\geq 2$, see \cite{IanoFletcher2000}.

\subsection{Analytic structure of weighted projective spaces}\label{analytic}
As regular projective spaces, the weighted complex projective spaces can also be equipped with an analytic structure. 
We  consider the decomposition of 
\[ \P_{\w, \C}^n  = U_0 \cup \ldots \cup U_n,\]
where   
\[
U_i=\{ \x  \in \P_{\w, \C}^n  :  \ x_i(\x) \neq 0\} \subset \P_{\w, \C}^n,
\]
for each  $0 \leq i \leq n$.   Then,  the map  $\tilde \psi_i : \C^n   \to U_i$,
\begin{equation}
	\begin{split}
		(x_0, \ldots , x_{i-1}, x_{i+1}  \ldots , x_n) & \to [x_0: \ldots : x_{i-1}: 1 :  x_{i+1} :  \ldots : x_n]_\w
	\end{split}
\end{equation}
is a surjective analytic map, but not a chart since it is not injective.  However, it induces the isomorphism  $\psi_i : \X  (q_i : q_0, \ldots , \widehat q_i, \ldots q_n )    \to U_i $, such as 
\[ 
[(x_0, \ldots , x_{i-1},   x_{i+1} ,  \ldots , x_n)]  \to [x_0: \ldots : x_{i-1}: 1 :  x_{i+1} :  \ldots : x_n]_\w,
\]
where $\X(q_i : q_0, \ldots , \widehat q_i, \ldots q_n )$ is the cyclic quotient space       of the action of      $\mu_{q_i}$ on $\C^n$ given by   
$\mu_{q_i} \times \C^n    \to \C^n$ such as 
\begin{small}
	\begin{equation}
		(\xi_i,  (x_0, \ldots, x_{i-1}, x_{i+1},  \ldots, x_n))   \to 
		(\xi_i^{q_0} x_0, \ldots ,  \xi_i^{q_{i-1}} x_{i-1},  \xi^{q_{i+1}}  x_{i+1} ,  \ldots , \xi_i^{q_{n}} x_n),
	\end{equation}
\end{small}
where $\xi_i \in \mu_{q_i}$.    Since the changes of charts are analytic, then $\P_{\w, \C}^n $ is an analytic space with cyclic quotient singularities;  
see \cite{MR3285300, ArtalBartolo2014} for details.
\subsection{Weighted Blow-ups}
Consider 
$
\widehat \C^{n+1}_{\w} \colonequals \left\{  (\x, [\u]_\w) \in \C^{n+1} \times \P_{\w, \C}^n    \, | \, \x \in \overline{[\u]}_\w       \right\},
$
where $\overline{[\u]}_\w$  denote the Zariski closure of $[\u]_\w$ and $\x \in \overline{[\u]}_\w$ means that there exists
$t \in \C$ satisfying  $x_i=t^{q_i}\cdot u_i$ for each $0\leq i \leq n.$
The natural projection map 
\begin{equation}\label{w-blowups}
	\pi_{\w} :   \widehat \C^{n+1}_{\w} \to \C^{n+1}
\end{equation}
is an isomorphism over  $\widehat \C^{n+1}_{\w} \setminus \pi^{-1}_{\w} (\mathbf{0})$ and the 
\textbf{exceptional divisor} 
$
E\colonequals \pi^{-1}_{\w} (\mathbf{0} )
$
is identified with $\P_{\w, \C}^n $. 
The space $\widehat \C^{n+1}_{\w} = \widehat {U}_0 \cup \ldots \cup \widehat {U}_n$  can be covered with  $(n+1)$ charts, where 
\[
\widehat {U}_i = \{ (\x, [\u]_\w) \in \C^{n+1} \times \P_{\w, \C}^n :  u_i \neq 0 \} \subset \widehat \C^{n+1} (\w).
\] 
However,   $\phi^i : \C^{n+1}  \to \widehat {U}_i $, 
\[
\x  \to \left( x_0^{q_0}, x_0^{q_1} x_1, \ldots , x_0^{q_n} x_n), [x_1 : \ldots , x_{i-1} : 1 : x_{i+1}: \ldots : x_n] \right), 
\]
are surjective, but not injective.  Indeed, we have that $\phi^i (\x) = \phi^i (\y)$  is and only if there exists $\xi\in \mu_{q_i}$ such that $y_i = \xi^{-1} x_i$ and $y_j = \xi^{q_j} x_j$ for $j\neq i$.  Hence, the map $\phi^i$ induces an isomorphism $\X (q_i : q_0, \ldots , q_{i-1}, -1, q_{i+1}, \ldots , q_n) \to  \widehat {U}_i$.

These charts are compatible with the ones of   $\P_{\w, \C}^n $.  In $ \widehat {U}_i$ the exceptional divisor is $\{ x_i =0 \}$ and the $i$-th chart of $\P_{\w, \C}^n $ is the quotient space 
\[
\X (q_i : q_0, \ldots , q_{i-1}, -1, q_{i+1}, \ldots , q_n).
\]
%

\begin{exa}[Case  $n=2$]
	Let  $\w = (q_0, q_1, q_2)$ be a  tuple of reduced weights, i.e., $\gcd (q_0, q_1, q_2) =1$ and 
	$
	\pi_\w : {\widehat \C}^3_{\w} \to \C^3,
	$
	be the weighted blow-up at the origin with respect to $\w$.    Then 
	$
	{\widehat \C}^3\iso  \widehat {U}_0 \cup  \widehat {U}_1 \cup  \widehat {U}_2
	$, 
	where
	\[
	\widehat {U}_0 \iso X(q_0:-1, q_1, q_2), \   \widehat {U}_1\iso X(q_1: q_0, -1, q_2), \  \widehat {U}_2\iso X(q_2: q_0, q_1, -1),
	\]
	and the charts are  given by 
	\[
	\begin{split}
		\psi^0:       \;     X(q_0: -1, q_1, q_2) & \rightarrow U_0, \qquad 
		[(x_0:x_1:x_2)]  \mapsto  \left( (x_0^{q_0}, x_0^{q_1} x_1, x_0^{q_2} x_2), [1: x_1: x_2]\right) \\
		\psi^1:    \;     X(q_1: q_0, -1,  q_2) & \rightarrow U_1, \qquad 
		[(x_0:x_1:x_2)]  \mapsto   \left( (x_1^{q_1} x_0, x_1^{q_1}, x_1^{q_2} x_2), [x_0: 1: x_2]\right) \\
		\psi^2:      \;   X(q_1: q_0, q_1, -1) & \rightarrow U_2, \qquad
		[(x_0:x_1:x_2)]  \mapsto   \left( (x_2^{q_2} x_0, x_2^{q_2} x_1, x_2^{q_2} ), [x_0: x_1: 1]   \right). \\
	\end{split}
	\]
	The exceptional divisor $\pi_\w^{-1}((0,0,0))$ is isomorphic to $\P_{\w, \C}^2$, which can be simplified by isomorphism   $\P_{\w, \C}^2 \iso \P_{\w', \C}^2$  given by 
	\[
	[x_0:x_1:x_2]\mapsto \left[ x_0^{\gcd(q_1,q_2)} : x_1^{\gcd(q_0,q_2)}:x_2^{\gcd(q_0,q_1)}  \right],
	\]
	where 
	\[
	\w'=\left( \frac{q_0}{\gcd(q_0, q_1) \cdot \gcd(q_0, q_2)}, \frac{q_1}{\gcd(q_0, q_1) \cdot \gcd(q_1, q_2)}, \frac{q_2}{\gcd(q_0, q_2) \cdot \gcd(q_1, q_2)}\right).  
	\]
\end{exa}

\section{Weighted heights}\label{sect-4}
In  \cite{b-g-sh} a height function was defined for weighted projective spaces $\P_{\w, k}^n $,  called  weighted height.  We briefly describe  basic definitions here.  To avoid confusion with projective heights we will use different notation than that of \cite{b-g-sh}.
We will follow the parallelism with   Weil heights by using $\wh$, $\lwh$ instead of  $H$, $h$.   $\P_{\w}^n  (k)$  denotes the set of $k$-rational points of $\P_{\w, k}^n$. 

\subsection{Weighted heights on   $\P_{\w, k}^n $}\label{WWHWPS}
Given  any $\x  \in \P_\w^n (k)$, the \textbf{multiplicative weighted  height} over $k$  is defined as  
\begin{equation}\label{def:height}
	\wh_k ( \x ) \colonequals \prod_{\nu \in M_k} \max   \left\{   |x_0|_\nu^{\frac{1}{q_0}} , \dots, |x_n|_\nu^{\frac{1}{q_n}}  \right\}
\end{equation}
and its  \textbf{logarithmic  weighted  height} (over $k$) as 
\begin{equation}\label{log-height}
	\lwh_k(\x) \colonequals \log \wh_k (\x)=    \sum_{\nu \in M_k}   \max_{0 \leq j \leq n}\left\{\frac{1}{q_j} \cdot  \log  |x_j|_\nu \right\}.
\end{equation}

In \cite[Prop.~ 1]{b-g-sh} it is shown that height functions $\wh_k ( \x )$ and hence $\lwh_k(\x)$ are independent of the choice of   coordinates of the point $\x$. Moreover, in   \cite[Prop.~ 5-ii]{b-g-sh}, it is proved that for any finite extension $K|k$ we have 
\[
\wh_k ( \x )^{[K:k]}=\wh_K ( \x ), \; \text{and hence} \;  {[K:k]} \cdot \lwh_k (\x) = \lwh_K (\x).
\]
Weighted    heights   can be interpreted  in terms of  Weil height on projective varieties using Veronese map defined by  \cref{veronese}.  
Assume  that   $\w=(q_0, \cdots , q_n)$   is reduced, well-formed and satisfies    $\gcd(m/q_0, \cdots , m/q_n)=1$,   where $m = \lcm (q_0, q_1 \cdots,  q_n)$.
Proof of the following can be found in \cite{b-g-sh}.
\begin{lem}
	Weighted height $\wh_k $   is given in terms of   projective height $H_k$   via 
	\begin{equation}\label{lwh1a} 
		\wh_k( \x )  = H_k  \left( {\phi_m} (\x) \right)^{\frac{1}{m}}   \; \text{ and } \; 
		\lwh_k(\x)  = \frac{1}{m} \cdot h_k \left( {\phi_m} (\x) \right),\\
	\end{equation}
	for all  $\x \in \P_{\w}^n (k)$,   where  $\phi_m$ is the Veronese map given in \cref{veronese}.
\end{lem}

The \textbf{absolute weighted height} on $\P_\w^n(\kk)$ is defined as
\begin{equation}\label{wgh}
	\begin{split}
		\wh  : \P_\w^n(\kk)  &    \to [0, \infty],  \\
		\x    &   \mapsto    \wh (\x) \colonequals   \wh_K(\x)^{1/[K:k]},
	\end{split}
\end{equation}
and the \textbf{absolute logarithmic weighted height}   on $\P_\w^n(\kk)$ is given by
\begin{equation}\label{wgh1}
	\begin{split}
		\lwh  : \P_\w^n(\kk)    &   \to [0, \infty ],   \\
		\x   \mapsto      &  \lwh (\x)\colonequals \frac{1}{[K:k]} \log   \wh_K(\x),
	\end{split}
\end{equation}
for which $K\subset \kk$ is a finite extension of $k$ containing   $k(\x),$ 
the
\textbf{field of   definition }   of    $\x$  defined by
\[
k(\x)\colonequals k \left( \frac{x_0^{1/q_0}}{x_i^{1/q_i}} , \cdots, 1, \cdots,     
\frac{x_n^{1/q_n }}{x_i^{1/q_i }} \right), 
\]
for some $x_i\neq 0$.
$  $
Notice that both of these height functions are independent of the choice of the field $K$; see \cite{b-g-sh}. For simplicity, we call   $\lwh (\x)$  the  \textbf{global  weighted height} on $\P_\w^n (\kk)$.

By   \cref{lwh1a},  for a field $K\subset \kk$ containing and $k(\x)$, we have:
\begin{lem}\label{lwh1b} 
	For all $\x \in  \P_\w^n(\kk) $, we have
	\begin{equation}
		\wh ( \x )   = H  \left( {\phi_m} (\x) \right)^{\frac{1}{m}}, \; \text{ and } \;     \lwh(\x)  = \frac{1}{m} \cdot h \left( {\phi_m} (\x) \right),
	\end{equation}
	where   $\phi_m$ is as in \cref{veronese},  $H (\cdot)$,  $h (\cdot)$ as in \cref{gh}, and $\wh (\cdot)$, $\lwh (\cdot)$ as in \cref{wgh}.
\end{lem}

\subsection{Cartier  and Weil divisors on weighted varieties}
Let $\X$  be a weighted  variety in $\P^n_{\w, k}$ over the field $k.$ 
The group of \textbf{Weil divisors} on $\X$  is a free Abelian group generated by weighted closed  subvarieties of codimension one on $\X.$  This group is denoted by  $\WeDiv_\w   (\X)$. 
The \textbf{support} of the divisor  $D=\sum_Y n_\Y \cdot \Y$  is the union of all codimension one weighted subvarieties $\Y$  such that $n_\Y \neq 0$, which is denoted by $\Supp   (D)$.
A divisor is said to be \textbf{effective} if every $n_\Y \geq 0$ for all codimension one subvarieties $\Y \subset \X$.    
We define $\ord_\Y: \cO_{\X,\Y} \setminus \{0\} \rightarrow \Z$  to be 
$$\ord_\Y(f)=\text{length}_{\cO_{\X,\Y}} \left(\frac{\cO_{\X,\Y}}{ \left\langle f\right\rangle } \right), $$
which is well defined since $\cO_{\X, \Y}$ is a local ring. 
Then, one can  extend $\ord_\Y$ to the fraction field $k_\w(\X)^*$ in the usual way.
%
The order function $\ord_\Y: k_\w(\X)^* \rightarrow \Z$ has the following properties:
\begin{enumerate}
	\item  $\ord_\Y(f\cdot g)=\ord_\Y (f)+\ord_\Y (g)$
	\item  For a fixed $f \in k_\w (\X)^*$ there are only finitely many $\Y$ such that $\ord_\Y \neq 0$.
	\item Let $f\in k_\w(\X)^*$. Then,  $f \in \cO_{\X,\Y}$ if and only if $\ord_\Y (f)\geq 0$.
	Similarly, $f \in \cO_{\X,\Y}^*$ if and only if $\ord_\Y (f)=0$.
	\item If $\X$ is weighted projective  variety and $f \in k_\w(\X)^*$, then  $f \in k^*$ if and only if  $\ord_\Y (f) \geq 0$ for all $\Y$; if and only if  $\ord_\Y (f)=0$ for all $\Y$.
\end{enumerate}
The divisor of any   $f \in k_\w(\X)^*$ is defined as 
\[ \div (f)=\sum_{\Y \subset \X } \ord_\Y (f)\cdot \Y \]
which is called a \textbf{principal divisor}.
Two divisors $D$ and $D'$ are said to be \textbf{linearly equivalent} if their difference is a principal divisor.
The  divisor of zeros and  divisor of poles  of $f$,  denoted by $(f)_0$ and $(f)_\infty$   respectively,   are
\[ (f)_0=\sum_{\ord_\Y(f)>0} \ord_\Y (f) \cdot \Y,  \ (f)_\infty=-\sum_{\ord_\Y <0} \ord_\Y (f) \cdot \Y\]
The \textbf{divisor class group} of $\X$ is the group of divisor classes modulo linear equivalence. This group is denoted by $\Cl_\w   (\X)$,  and  $\Cl ( \P^n_{\w,k})$  for $\X =  \P^n_{\w,k}$.

A \textbf{Cartier divisor} on a weighted variety $\X$ is an equivalence class of collection of pairs $(U_i, f_i)_{i\in I}$   satisfying the following conditions:

\begin{enumerate}[\upshape(i)]
	\item The $U_i$ are affine weighted open  sets that cover $\X$.
	
	\item The $f_i$ are non zero rational functions, $f_i \in k_\w(U_i)^*=k_\w(\X)^*$.
	
	\item $\frac{f_i}{f_j} \in \cO_\X{(U_i\cap U_j)}^*$, so $\frac{f_i}{f_j}$ has no poles or zeros on $U_i\cap U_j$.
	
\end{enumerate}

\noindent Two Cartier divisors $\{(U_i, f_i)| i \in I\}$ and $\{(V_j, g_j)| j \in J\}$ are equivalent if  for all $i \in I$ and $j \in J$ we have 
\[
\frac{f_i}{g_j}\in \cO_\X(U_i\cap V_j)^*.
\]
The \textbf{sum of two Cartier divisors} is
\[\{(U_i, f_i)| i \in I\}+\{V_j, g_j)| j \in J\}=\{(U_i\cap V_j , f_i \, g_j)| (i,j) \in I\times J \}.\]  
The Cartier divisors with this operation on a weighted variety $\X$ form a group that we denote it by $\CaDiv_\w (\X)$.  The \textbf{support} of a Cartier divisor is the set of zeros and poles of  $f_i$, which is denoted by $\Supp(D)$.  A Cartier divisor is said to be \textbf{effective} or \textbf{positive} if it can be defined by a collection $\{(U_i, f_i)| i \in I\}$  such that every $f_i \in \cO_\X(U_i).$  
For a given  $f\in  k_\w(\X)^*$, the divisor  $\div(f)=\{(\X,f)\}$ is called a \textbf{principal Cartier divisor}.  Two Cartier divisors  are \textbf{linearly equivalent}  if their difference is a principal divisor.  The group of Cartier divisors classes modulo linear equivalence is called \textbf{Picard group} of a weighted variety $\X$ and is denoted by $\Pic_\w (\X)$. 
In the case $\X= \P_{\w, k}^n $, we   write  $\Pic(\P_{\w ,k}^n)$. 
A Cartier divisor $D$ on a weighted variety $\X$ is said to be \textbf{ample} or \textbf{big} if the corresponding line bundle $\cO(D)$ is ample or big, respectively.

For  $\X= \P_{\w, k}^n $ with reduced weights $\w$, 	 in \cite[Sections 5, 6]{AlAmrani1989},
it is proved  that  the following maps
\begin{align}
	\begin{split}
		\Z \rightarrow \Cl(\X), \\
		1 \mapsto \cO_\X (1),  \\
	\end{split}
	&
	\begin{split}
		\Z \rightarrow \Pic(\X), \\
		1 \mapsto \cO_\X (m),  \\
	\end{split}
	& 
	\begin{split}
		\\
		m=\lcm(q_0, \cdots, q_n),
	\end{split}
\end{align}
induce  the following isomorphism $\Cl(\X) \iso \Z, $ and $ \Pic(\X) \iso \Z, $  respectively.
Furthermore,  $\cO_\X (a)$ is not necessarily an invertible sheaf for any given integer $a\in \Z$. 
However,   by \cite[Lem. 1.3]{Mori1975}, the sheaf $\cO_{\X} (m)$  with $ m=\lcm (q_0, \cdots, q_n)$ 
is  ample and invertible, and for $a, b \in \Z$ we have 
\[
\cO_{\X} (a) \otimes \cO_{\X} (m)^{\otimes b} \iso \cO_{\X} (a+b m).
\]
In \cite[Thm. 4B. 7]{Beltrametti1986}, it is proved that  $ \cO_{\P_{\w, k}^n} (m)$ 
is ample  and there is $c \in \Z$  such that  $ \cO_{\P_{\w, k}^n} (c m)$ is very ample. 
Furthermore, the sheaf  $ \cO_{\P_{\w, k}^n} (a)$ is   coherent  and Cohen-Macaulay for any  $a \in \Z$.
If  $ \cO_{\P_{\w, k}^n} (a) \neq 0, $ then it is reflexive of rank $1$ by  \cite[Cor. 5.8]{Beltrametti1986}.

Following \cite{Mori1975}, we   define the \emph{weak projective space} over any field $k$ as follows:
\begin{defi}\label{wps}
	The complement of  $\sing(\P_{\w, k}^n )$  in $\X =\P_{\w, k}^n $ is called the  \textbf{weak projective space} over $k$, which  is a smooth weighted subvariety, denoted by 
	\begin{equation}
		\wP_{\w, k}^n :=\P_{\w, k}^n  \setminus \sing(\P_{\w, k}^n ).
	\end{equation}
\end{defi}

By  \cite[Prop. 1.1]{Mori1975}, the sheaf $\cO_\X (1)$ is locally free on $\wP_{\w, k}^n $.
Hence, defining  
\[
\cO_{\wP_{\w, k}^n} (1)\colonequals \cO_{\P_{\w, k}^n} (1)|_{\wP_{\w, k}^n },
\]
one can see that $\wP_{\w, k}^n$   is the largest open set $U \subset\P_{\w, k}^n$  such that  $\cO_{\P_{\w, k}^n} (1)|_U$ is an invertible sheaf on $U$ and 
\[
\left( \cO_{\P_{\w, k}^n} (1)|_U\right) ^{\otimes a} \iso \cO_{\P_{\w, k}^n} (a)|_U
\]
for any $a\in \Z$ by \cite[Thm. 1.7]{Mori1975}.
Furthermore, we have  $\Pic_\w (\wP_{\w, k}^n) \iso \Z$ and it is  generated by $\cO_{\wP_{\w, k}^n} (1)$.

For any (weighted)  projective variety $\X$ of dimension $\dim(\X)=d$ over $k$, we denote by $\Omega_\X^i$  the sheaf of   $i$-th regular differential forms on $\X$, and  $\omega_\X =\Omega_\X^d$ the \textbf{canonical  sheaf} of $\X$.
By \cite[Prop. 2.3]{Mori1975}, the canonical sheaf of $\wP_{\w, k}^n$ is
\[
\omega_{\wP_{\w, k}^n}\iso \cO_{\wP_{\w, k}^n} (-  \tilde{q}),
\]  
where $\tilde{q} = q_0 + q_1+\cdots+ q_n$,  by \cite[Prop. 2.3]{Mori1975}.

We also denote by $\omega_\X^0$ the  \textbf{dualizing sheaf} of $\X$.
If  $\X$ is a  nonsingular or more generally  normal (weighted) projective variety, then $\omega_\X^0=\omega_\X$. 
Otherwise, we let  $\Wc=\X \backslash \sing(\X)$ and consider the canonical embedding $j: \Wc \rightarrow \X$. 
Then,   if $\text{Codim}_\X (\X-\Wc ) \geq 2,$ then 
\[
\omega_\X^0 =j_*\ \omega_\Wc^0 =j_*\ \omega_\Wc.
\]
In the case $\X=\P^n_{\w, k}$,     since it is normal and Cohen-Macaualy and $\Wc = \wP_{\w, k}^n $, so by \cite[Cor. 6B.8]{Beltrametti1986} one has $\omega_{\P^n_{\w, k}}^0 \iso \cO_{\P^n_{\w, k}}(-\tilde{q}).$

\subsection{Local weighted heights}\label{LWH}
We assume that $\X$ is a weighted variety defined  over $k$ in   $\P^n_{\w ,\kk}$,  where $k \subset  \kk $ and $\w=(q_0,\cdots, q_n)$.
If $\X$ is a  weighted affine variety with coordinates $x_0, x_1, \cdots, x_n$, then a set  $E\subset \X \times M$ is called   a \textbf{weighted  affine $M_k$-bounded set }  if there  is an  $M_k$-bounded constant function  $\gamma$ such that 
\[\displaystyle 
|x_i(\x)|_v^{\frac{m}{q_i}}   \leq  e^{\gamma (v)}, \  0 \leq i \leq n \  \text{and} \  (\x, v) \in E.
\]
We note that this definition is independent of choice of the coordinates $x_i$'s on $\X$.  Moreover,  any finite union of  weighted  affine $M$-bounded sets is again a   weighted  affine $M$-bounded.

For an arbitrary variety $\X$, we say that  $E\subset \X \times M$ is a \textbf{weighted  $M_k$-bounded set} if there exists a finite cover  $U_i's$ of weighted  affine open subsets of $\X$ and the weighted   $M_k$-bounded sets $E_i \subset U_i\times M$  such that $E=\bigcup E_i$.
A function 
\[
\lambda: \X \times M \rightarrow \R
\]
is called a \textbf{locally weighted  $M_k$-bounded above} if for every weighted  $M_k$ bounded subset $E \subset \X \times M,$ there exists an $M_k$-constant  $\gamma$ such that  $\lambda (\x, v ) \geq \gamma(v)$ holds for $(\x, v) \in E$.  The  \textbf{locally weighted  $M_k$-bounded below}  and \textbf{locally weighted  $M_k$-bounded } functions are  defined similarly.

\begin{exa}
	For example,  let $\X=\P_{\w, \kk}^n$ and consider the finite cover of 
	affine open sets $\{ (U_i, x_i )\}  $ and $\gamma \equiv 0$.
	Moreover, for $0\leq i \leq n, $ the following sets are weighted $M_k$-bounded:
	\begin{equation}\label{En}
		\widetilde{E}_i=\left\lbrace  (\x,  v) \in \X \times M:  \text{and}\  
		\left| \frac{x_0^{\frac{m}{q_0}}}{x_i^{\frac{m}{q_i}}} \right|_{v}  \leq 1, \cdots, 
		\left| \frac{x_n^{\frac{m}{q_n}}}{x_i^{\frac{m}{q_i}}} \right|_{v}  \leq 1 \right\rbrace .
	\end{equation}
	Thus  $ \X=\P_{\w, \kk}^n$ is a weighted $M_k$-bounded set, since it is covered by $\widetilde{E}_i'$s. 
\end{exa}

Let $\cL$ be a line bundle on a weighted variety $\X$ defined over $k$.  A  \textbf{weighted $M$-metric}  on  $\cL$ is a norm $\|\cdot \|=\left( \|\cdot \|_v \right) $  such that for  each $v \in M$,  extending $v|_k \in M_k,$ and each fiber $\cL_\x$ with $\x \in \X$ assigns a   function $\|\cdot \|_v: \cL_\x \rightarrow \R_{\geq 0}$, not identically equal to zero, satisfying the following:
\begin{itemize}
	\item $\|\lambda \cdot  \xi\|_v = |\lambda|_v  \cdot \|  \xi \|_v$     for $\lambda \in \kk$ and $\xi \in \cL_\x$.
	\item If $ w_1, w_2 \in M$ agree on the residue field  $k(\x)$, then $\|\cdot \|_{w_1}=\|\cdot \|_{w_2}$ on $\cL_\x(k(\x))$.
\end{itemize}

A weighted  $M$-metric on $\cL$ is called  \textbf{locally  weighted  $M$-bounded} if for  section $g\in \cO_\X(U)$ on an open
set $U \subseteq \X$, the function 
\[
(\x, v) \mapsto \log \|g(\x) \|_{v}
\]
on $U\times M$ is locally weighted   $M_k$-bounded.
We say that   $\cL$ is  a \textbf{weighted $M$-metrized line bundle} on $\X$ if $\cL$ is  
equipped with a  weighted $M$-metric  $\|\cdot \|=\left( \|\cdot \|_v \right).$ 

Next  we  show that  there exist a locally bounded weighted $M$-metric on any line bundle on the weighted variety $\X$.
\begin{prop}\label{Wmetric}
	Any line bundle $\cL$ on a weighted variety $\X \subseteq \P_{\w, \kk}^n$ defined over $k$ admits a locally bounded weighted  $M$-metric.\footnote{We thank Min Ru for clarifying  some details in the proof of \cref{Wmetric} by indicating    \cite[B2.2.10 and B2.2.11]{Ru2021a}.}
\end{prop}
\proof 
First we assume that   $\X=\P_{\w, \kk}^n$ and $\cL=\cO_\X(m)$, where $m=\lcm (q_0, q_1, \cdots, q_n)$.
Then, one can define an $M$-metric by  letting
\begin{equation}\label{eq2c}
	\displaystyle\|\ell(\x) \|_v =\frac{|\ell(\x)|_v}{ \max_{i}^{} |x_i|_v^{\frac{m}{q_i}}},
\end{equation}
for each $v\in M$, $\x \in \X$ and    a global section $\ell\in \cO_\X(m)$   given by
$$\ell= a_0 x_0^{\frac{m}{q_0 }} + a_1 x_1^{\frac{m}{q_1 }}+ \cdots+ a_1 x_1^{\frac{m}{q_1 }}.$$   
It is well-defined  on  $\cL$, and  on the set $U_i=\{x_i \neq 0\}$  we have  
$$\left\|x_i^{\frac{m}{q_i}}(\x) \right\|_v =\frac{\left| x_i^{\frac{m}{q_i}}(\x)\right|_v}{ \max_{i}^{} |x_i|_v^{\frac{m}{q_i}}} \leq 1.$$ 
Moreover,  the functions $\left| \frac{x_j^{m/q_j}}{x_i^{m/q_i}}\right|_v$
are bounded by an $M_k$-constant  on the bounded sets $\widetilde{E}_i$ defined by \cref{En}.
Thus,  $\log \left\|  x_i^{\frac{m}{q_i}}(\x)\right\|_v$ are bounded below for all indexes, and hence  \cref{eq2c}  gives the  desired locally bounded weighted  $M$-metric.  

Next, we assume that  $\X \subseteq \P_{\w, \kk}^n$ is a weighted projective variety  
and $\cL=\cO_\X(D)$, where $D$ is an effective Cartier divisor on $\X$ both defined over $k$.
In this case,  
$\cL$ can be written as     $ \cL=\M_1\otimes \M_2^{-1}, $ where $\M_1$ and $\M_2$ are base point free line bundles on $\X$. Now, we choose  generating  global functions $s_1, \cdots, s_{n_1}$ of $\M_1$, and  $t_1, \cdots, t_{n_2}$ of $\M_2$.  
Then, for $\x \not \in  \text{Supp}(D) $, the desired   locally bounded weighted $M$-metric on $\cL$  is given by
\begin{equation}\label{eq2d}
	\|g_D (\x) \|_v= \max_{1\leq i \leq n_1} \min_{1\leq j \leq n_2} 
	\left\| \frac{s_i g_D}{t_j } (\x) \right\|_v,
\end{equation} 
where $v\in M_k$ and $g_D $ is a section of  $\cL=\cO_\X(D)$ with $D=div(g_D)$. 
One can show   this metric is uniquely determined and independent of choices  $\M_1$, $\M_2$, and their generating sections as  \cite[B2.2.10 and B2.2.11]{Ru2021a}.
We notice that if  $\X =\P_{\w, \kk}^n$ and  $\cL=\cO_\X(m)$,  then  \cref{eq2d}   will be same as \cref{eq2c} by considering $\M_1=\cL$  and $\M_2$ trivial line bundle and $t_i=x_i^{m/q_i}$ for $0\leq i \leq n$ and $g_D \in \cO_\X(m).$ 

Finally, for an arbitrary weighted variety $\X$, first we cover it by  finitely many open affine
sets $U_i$ such that on each $U_i$ the line bundle $\cL$ is trivialized with a non-vanishing section $g_i$.
Letting $p_{j,t}$ be the coordinates on $U_j$  with $p_{j0}=1$, one can find
constants $C$ and  $\gamma$ (not depending on   $i$ and $j$)  such that
\[
\left| \frac{g_i(\x)}{g_j(\x)}\right|_v   \leq   C \cdot \max_{t}^{}   |p_{jt}|_v^{\gamma},
\]
and hence for $\x \in U_i\cap U_j$ we have 
\[
\left |g_{ji}(\x) \right |_v =  \left| \frac{g_j(\x)}{g_i(\x)}\right|_v  \geq   \frac{1} { C \cdot \max_{t}^{}   |p_{jt}|_v^{\gamma}}.
\]
Thus,   for   $\x \in U_i$, defining 
\begin{equation}\label{eq2d1}
	\| g_i(\x) \|_v= \max_{t}^{}    \min_{\{j:\ \x \in U_j\} }^{}   |p_{jt}|_v^{\gamma} \cdot \left| \frac{g_i(\x)}{g_j(\x)}\right|_v 
\end{equation}
we obtain the desired   locally bounded weighted $M$-metric of $\cL$  on $U_i$, which is independent of the choice of  transition functions $g_{ji}=g_j/g_i$ over $U_i\cap U_j$. Using the cocycle rule $g_{e j}= g_{ei} g_{ij}$, for every 
$\x \in U_e\cap U_i$, we have $$\left\|g_e(\x) \right\|_v=\left|g_{e i} (\x) \right|_v \cdot \left\| g_{i}(\x)\right\|_v.$$
Therefore,  \cref{eq2d1} provides a well-defined $M$-metric of $\cL$ on $\X$.
By a similar argument as in the end of proof of \cite[Prop. 2.7.5]{bombieri} or \cite[B2.2.10]{Ru2021a},
one can see that this is a locally bounded weightd metric.
\qed

We denote by    $\widehat{\Pic_\w(\X)}$ the  group of   isometric classes of 
pairs  $\tilde{\cL}=(\cL, \|\cdot \|)$.
As in the usual case,  given any morphisms  
$\phi: \X' \rightarrow \X$
of weighted varieties over $k$,   and 
$\widehat{\cL}= (\cL , \| \cdot  \|) \in  \widehat{\Pic_\w(\X)},$
the pull-back of $\widehat{\cL}$ by $\phi$ is defined as 
$ \widehat{\phi^*(\cL)} =  ( \phi^*(\cL), \| \cdot  \|') ,$
such that 
\begin{equation}
	\|  \phi^*(g) (\x) \|'= \| g (\phi(\x))\|  \ (\x \in \X'),
\end{equation}
for  any open subset $U$ of $\X$ containing $\phi(\x)$ and  $g \in \cO_\X(U)$.

The pull-back induces a group homomorphism  between
$\widehat{\Pic_\w(\X)}$ and $\widehat{\Pic_\w(\X')}$. 
Under this homomorphism, any  locally bounded  weighted $M$-metrized 
line bundles remain locally  bounded  weighted $M$-metrized.
Now  we can define the weighted local Weil heights on a  variety $\X$ in  $\P_{\w, \kk}^n$ as follows: 
Given any   Cartier divisor $D=\{ (U_i, f_i)\}$  on $\X$,  we let  $\cL_D=\cO_\X(D)$ be  the line bundle of regular functions on $D$. It can be constructed by gluing 
\[
\cO_\X(D)|_{U_i}= f_i^{-1} \cO_\X( U_i)
\]
and $1$ becomes a canonical invertible meromorphic section of $\cL_D$, which is denoted by $g_D$.
Thus, by \cref{Wmetric},  we can equip
$\cL_D$   with a weighted  locally bounded $M$-metric $\| \cdot \|$,  determined    by the max-min method  in   proof of \cref{Wmetric}, and denote it  by $\widehat{D}=\left( \cL_D, \| \cdot \|\right).$

\begin{defi}
	Given  $\nu \in M_k$, we define the \textbf{local weighted height  $\il_{\widehat{D}}(-,\nu)$ with respect to} $\widehat{D}$ on the weighted variety  $\X$ as 
	\begin{equation}\label{wlwh1}
		\il_{\widehat{D}}(\x, \nu):=- \log \| g_D(\x)\|_v 
	\end{equation} 
	for  $\x\in \X \backslash  \Supp(D),$	where   $v\in M$ such that $\nu=v|_k.$
\end{defi}

We note that the  local weighted height  $\il_{\widehat{D}}(-,\nu)$  is well defined because the norm 
$\| \cdot\|$ is well-defined by  its construction as it explained in proof of \cref{Wmetric}.

Here, we have the fundamental  properties of the local   weighted heights. 

\begin{thm}[Weighted local Weil height machinery]  
	\label{WLWHM}
	For each of $\nu \in M_k$,   fix $v\in M$ such that $\nu=v|_k$. Suppose that $\X$ is a weighted variety  defined over $k$  and   ${\widehat{D}}, {\widehat{D}}_1, {\widehat{D}}_2\in \widehat{\Pic_\w(\X)}$.    Then:
	\begin{enumerate}[\upshape(i), nolistsep]
		\item  \textbf{Additivity:}   
		For    $\x \not \in \Supp(D_1) \cup  \Supp(D_2)$, we have 
		\[ \il_{\widehat{D_1+D_2}}(\x,\nu)=  \il_{{\widehat{D}}_1 }(\x,\nu) + \il_{\widehat{D}_2}(\x,\nu).\]
		
		\item   \textbf{Functoriality:} If $\phi: \X' \rightarrow \X$ is a morphism of weighted varieties  defined over $k$ such that $\phi(\X')  \cap  \Supp(D) =\empty$, then 
		\[\il_{\phi^*(\widehat{D})}(\x', \nu)=\il_{\widehat{D}}(\phi(\x'),\nu )  \ \text{for} \ \x' \in  \X' \backslash  \phi^{*}(D).\]
		
		\item    \textbf{Boundedness from below:}  If $D$ is effective and  $\X$ is weighted $M_k$-bounded  projective variety,   then there exists an $M_k$-constant function $\gamma$ such that 
		\[\il_{\widehat{D}}(\x, \nu) \geq \gamma(\nu)\ \text{for} \ \x \in  \X \backslash  \Supp(D).\]
		
		\item \textbf{Normalization:}  If $\X=\P_{\w, \kk}^n$ and $D$ is a hyperplane  defined by  $\ell\in \cO_\X(m)$, with $m=\lcm (q_0, q_1, \cdots, q_n)$, then 
		
		\begin{equation}\label{wlwh2}
			\il_{\widehat{D}}(\x, \nu) = -\log \frac{|\ell(\x)|_v}{ \max_{i}^{} \left| x_i\right|_v^{\frac{m}{q_i}}}  \ \text{for} \ \x\in  \X \backslash  \Supp(D). 
		\end{equation}

		\item   \textbf{Principal divisor:} If $D=\div(f) $ for some nonzero  $f\in \cO_\X(D)$ with $\deg(f)=d$, then 
		\begin{equation}\label{whype2}
			\il_{\widehat{D}}(\x, \nu)= -\log \frac{|f(\x)|_v}{ \max_{i}^{} \left| x_i\right|_v^{\frac{ d}{q_i}}}, \ \text{for} \ \x\in  \X \backslash  \Supp(D), 
		\end{equation}
		by letting $\|1 \|_v=|1|_v$ on $\cO_\X(D)$ for $v\in M$ over $\nu \in M_k$.

		\item \textbf{Uniqueness:} If $\X$ is weighted $M_k$-bounded, $\|\cdot \|'_v$ is another weighted $M_k$-bounded metric on $\cL_D$ and  $\il'_{\widehat{D}}$ is the resulting  local weighted Weil height respect to $\left( \cL_D, \| \cdot \|'\right)$,  then
		$$\il_{\widehat{D}}(\x, \nu) = \il'_{\widehat{D}} (\x, \nu) + O(1).$$
		
		\item \textbf{Base change:}  If $K|k$ is a finite field extension and $u\in M_K$ over some $v \in M_k$, then
		\[\il_{\widehat{D}}(\x,\nu) =\frac{1}{[K_u: k_\nu]} \il_{\widehat{D'}}(\x', u), \ \text{for} \ \x' \in  \X' \backslash  \Supp(D'),\]
		where $\X'=\X \otimes_k K$ and $\x' \in  \X' $    corresponds to $\x \in \X(k)$, and $D' \CaDiv(\X')$ correspond to $D$.
		
		\item \textbf{Max-Min:}  There are  positive integers $n_1$ and $n_2$, and  nonzero rational functions  $f_{ij} $ on $\X$ for $i=0, \cdots, n_1$ and $j=0, \cdots, n_2$ such that
		\[\il_{\widehat{D}}(\x,\nu) =\max_{0 \leq i\leq n_1}^{} \min_{0\leq j\leq n_2}^{}  \log  \left| f_{ij} (\x) \right|_\nu.\]
	\end{enumerate}
\end{thm}
\proof
The proofs are almost straightforward and similar to   proof of the Weil local heights on projective heights.
\begin{enumerate}[\upshape(i), nolistsep]
	\item   Using the product of weighted $M$-metrics from $\cO_\X(D_1)$ and $\cO_\X(D_2)$  on $\cO_\X(D_1+D_2)$, and 
	$g_{D_1+ D_2}=g_{D_1} \otimes g_{D_2}$, we have
	\[\|g_{D_1+ D_2} \|_\nu =\|g_{D_1} \otimes g_{D_2}\|_\nu=\|g_{D_1} \|_\nu \cdot \|g_{D_2}\|_\nu,\]
	which implies the desired equality by taking logarithm from both sides. 
	
	\item The functoriality is a direct consequence of the functoriality of the weighted $M$-metrics $\| \cdot \| = (\| \cdot \|_v)$, i.e., 
	$\|\phi^*(g_D)(\x)\|=\|g_D(\phi(\x))\| $ for all $v\in M$.
	
	\item Note that the rational function $g_D$ is defined everywhere for any effective divisor $D$.
	Then,  on bounded sets inside an affine open set $U$ of $\X$
	where $\cO_\X(D)$ is trivial  and so all global sections can be identified non-canonically  as regular functions,
	$|g_D(\x)|_v $ and is bounded above   by an $M_k$-constant. 
	This implies that    $\il_D(\x, \nu)$ is  bounded below   by an $M_k$-constant.

	\item  A locally  $M_k$-bounded metric on $\cO_\X(D)\iso\cO_{\P_{\w, \kk}^n}(m)$ is given by \cref{eq2c} and  
	hence $g_D =\ell $ is defined  away from the hyperplane $D$.  Given any $\nu \in M_k$ and fixing  $v \in M$ such that $\nu=v|_k$, one can get  \eqref{wlwh2}
	by taking logarithm.
	
	\item  For a divisor $D=\div (f)$ with $\deg(f)=d$, we have $\cO_\X(D)= f^{-1} \cO_\X$ and $g_D=f$ whenever $f$ is defined.   
	Hence, for any $v$ over $\nu$, we have 
	\[
	\| f(\x)\|_v= -\frac{|f(\x)|_v}{ \max_{i}^{} \left| x_i\right|_v^{\frac{d}{q_i}}},
	\]
	By taking logarithm, this implies  \cref{whype2} as desired, 
	
	\item Using (i) with $\widehat{D} = \widehat{D} + (0)$ where $\widehat{D}$ on the left hand side is endowed with $\| \cdot \|'$, then 
	\[
	\il_{\widehat{D}}(\x, \nu) - \il'_{\widehat{D}}(\x, \nu)
	\]
	is  the  logarithm of norm of $1$ 
	with the locally bounded metric  $\| \cdot \|_v/\| \cdot \|'_v$ on $\cO(\X)$.
	Since $1$ is a global nowhere-vanishing section,    by the definition,  we have $\il_{\widehat{D}}(\x, \nu) = \il'_{\widehat{D}} (\x, \nu) + O(1)$.
	\item  Since $|\cdot |_v=|\cdot |_u^{1/[K_u: k_v]}$ for  $u\in M_K$ over $v \in M_k$, so 
	$\| \cdot \|_\nu=\|\cdot \|_u^{1/[K_u: k_v]}$ and hence the desired equality.
	\item  By linearity of the  both sides of  equality, 
	\[\il_{\widehat{D}}(\x,\nu) =  \max_{0 \leq i\leq n_1}^{} \min_{0\leq j\leq n_2}^{} 
	\log  \left| f_{ij} (\x) \right|_\nu\]
	and  the proof of   \cref{Wmetric},  it is enough to consider  $\widehat{D}$ such that $\cO_\X(D)\iso \cO_\X (m)$. 
	In this case, the existence of  $f_{ij}$'s is clear by the proof of \cref{eq2c}.
\end{enumerate}
\qed

\subsection{Global weighted heights}   \label{WGWHVPS}
Now, we assume  $\X \subseteq \P^n_\w(\kk)$  is a   weighted variety and
consider   $\widehat{\cL}=(\cL, \| \cdot  \|) \in \widehat{\Pic_\w(\X)}$.  
Given $\x \in \X$,  let $K$ be a finite extension of $k$ containing $k(\x)$. 
For each $u\in M_K$, we choose a place  $v\in M$  over $u$   and define  $\|\cdot \|_u\colonequals \|\cdot \|_{v}^{1/[K:k]}$ on $\cL_{\x}(k(\x))$.  By the second  condition of a weighted $M$-metric, one can see that it is independent of the choice of $v\in M$.   We let $g$ be an invertible regular function of $\cL$  with $\x \not \in \Supp(\cL_g)$ where $\cL_g=\div(g)$.  
Note that such function exists because there is an open dense trivialization in a neighborhood of the point $\x$.
Then,   we have the weighted $M$-metrized  line bundle $\widehat{\cL_g}=\left( \cO_\X(L_g),  ( \|\cdot \|_u ) \right)  \in \widehat{\Pic_\w(\X)}$.

The \textbf{global weighted height $\hn_{\widehat{\cL}}(\x)$ }   with respect to   $\widehat{\cL}$  is defined by
\begin{equation}\label{gwh0}
	\hn_{\widehat{\cL}}(\x)         \colonequals \sum_{u\in M_K}^{} \il_{\widehat{\cL_g}} (\x, u),
\end{equation}
where   $\il_{\widehat{\cL_g}} (\x, u) = -\log \| g(\x)\|_u $ assuming $v|_k=u$.   It is easy to check that these definitions are independent of the choice of  field $K$ and  regular function $g$.

\begin{exa}
	Let $\X=\P_{\w, \kk}^n$,  $D=\div ( x_0^{1/q_i} ) $, and $\cL=\cO(D)$.  Then, one has  $\hn(\x) =\hn_{\widehat{\cL}}(\x)$,  where $\hn(\x)$ is the  global weighted height on $\P_{\w, \kk}^n$ given by \cref{wgh}.
	
	Indeed, if  $K=k(\x)$ and $u \in M_K$ over $\nu \in M_k$, \cref{wlwh2} becomes 
	\begin{equation}\label{s-1}
		\il_{\widehat{D}}(\x, u)= -  \log \frac{ \left| x_0^\frac{1}{q_0} \right|_u }{ \max_{i}^{} \left| x_i^{\frac{1}{q_i}}\right|_u}, \ \text{for} \ \x\in  \X \backslash  \Supp(D). 
	\end{equation}
	Since $\il_{\widehat{\cL_{x_0}}} (\x, u)$ and $ \il_{\widehat{D}}(\x, u)$ are same local height, 
	we have 
	\begin{align*}
		\hn_{\widehat{\cL}}(\x)&=\sum_{u \in M_K}^{} \il_{\widehat{\cL_{x_0}}} (\x, u) 
		= 
		\sum_{u\in M_K}^{} - \log \frac{ \left| x_0^\frac{1}{q_0} \right|_u }{ \max_{i}^{} \left| x_i^{\frac{1}{q_i}}\right|_u} \\
		&  = \sum_{u\in M_K}^{} \frac{1}{q_i} \log \ \max_{i}^{} \left| x_i\right|_u -  
		\frac{1}{q_0}	\sum_{u\in M_K}^{} \log \left| x_0 \right|_u.
	\end{align*}
	The last term vanishes by  product formula and  using   \cref{WLWHM} (vi),  
	we  have	
	\begin{align*}
		\hn_{\widehat{\cL}}(\x) & =  \sum_{u\in M_K}^{}  \frac{1}{q_i}\log \ \max_{i}^{} \left| x_i \right|_v 
		=    \sum_{v\in M_k \ u| v}^{}  \frac{1 }{[K_u: k_v] q_i}\log \ \max_{i}^{} \left| x_i \right|_v \\
		&  =\frac{1}{[K : k]}  \cdot \sum_{v\in M_k}^{}  \ \max_{i}^{}  \left\lbrace \frac{1}{q_i} \cdot \log  \left| x_i \right|_v\right\rbrace  =\hn(\x).
	\end{align*}
\end{exa}

The above example shows the normalization property of the weighted global Weil height function, 
and their other essential properties are given  by the following theorem.  
\begin{thm}[Global weighted height machinery]  \label{WGWHM}
	Let $\X$ be a weighted variety and consider  $\widehat{\cL}, \widehat{\cL}_1, $ and  $\widehat{\cL}_2 \in \widehat{\Pic(\X)}$.
	%
	\begin{enumerate}[\upshape(i), nolistsep]
		\item[(i)]  \textbf{Independence (a):}  $\hn_{\widehat{\cL}} $ depends  only on the isometry class of  $\widehat{\cL}$, i.e, if $\widehat{\cL}_1 $and $ \widehat{\cL}_2 $ are  isometric pairs, then $\hn_{\widehat{\cL}_1} =\hn_{\widehat{\cL}_2}.$
		\item[(ii)] \textbf{Independence (b):}  If $\X$ is a complete weighted variety or generally $M$-bounded, then  $\hn_{\widehat{\cL}}$ does not depend on the choice of  weighted  locally bounded $M$-metrics up to a locally $M$-bounded constant function.
		\item[(iii)] \textbf{Additivity:} For   any $\x \in \X$, we have $\hn_{\widehat{\cL}_1 \otimes \widehat{\cL}_2  }(\x)=\hn_{\widehat{\cL}_1} (\x)+\hn_{\widehat{\cL}_2}(\x).$
		\item [(iv)]  \textbf{Functoriality:} If $\phi: \X' \rightarrow \X$ is a morphism  of weighted varieties  over $k$, then 
		\[\hn_{\phi^*(\widehat{\cL})}(\x)=\hn_{\widehat{\cL}}(\phi(\x))  \ \text{for} \ \x \in \X.\]

		\item[(v)]  \textbf{Base change:}  If $K|k$ is a finite field extension, then
		\[\hn_{\widehat{D}}(\x) =\frac{1}{[K: k]} \hn_{\widehat{D'}}(\x'), \ \text{for} \ \x' \in  \X' \backslash  \Supp(D').\]
		where $\X'=\X \otimes_k K$ and $\x' \in  \X'$    corresponds to $\x \in \X(k)$, and $D' \CaDiv(\X')$ correspond to $D$.
		
		\item [(vi)]   
		If  $\widehat{\cL}$ is a line bundle on $\X$, generated by its global sections, then   $\hn_{\widehat{\cL}}(\x)$ 
		is bounded from below for all $\x \in \X(\kk)$, by a constant depending  on $\widehat{\cL}$.

	\end{enumerate}
\end{thm}

\proof 
The proof is essentially similar to the proof of  \cref{GWHM}.
The part (i) is obvious  by definitions.
One may conclude part (ii) using  (iii) of  \cref{WLWHM} and the definitions. 
The part (iii) comes  from    (i) of  \cref{WLWHM}, and 
(iv)  is a consequence of    (ii) of  \cref{WLWHM}.  			
The   part (v) comes by   (vii) of  \cref{WLWHM}, and
(vi) is a result of  (iii) of   \cref{WLWHM}.

\qed

\subsection{Weighted  local and global heights for closed subschemes}\label{sect-4.6}
The local and global heights for closed subschemes of projective varieties are introduced in \cite{Silverman1987}. Here, we   develop them to the closed subschems of weighted projective varieties. 

Fix   a weighted projective variety $\X $ in $\P_{\w, \kk}^n$ defined over $k$, i.e., an integral and separated subscheme of finite type.
Given any closed subscheme $\Y$ of $\X$ over $k$, we let $\cI_{\Y}$ denotes the corresponding 
sheaf of ideals generated by  $f_1,\cdots, f_r \in k_\w[x_0, x_1, \cdots, x_n]$ with $\deg(f_j)=d_j$ for $j=1,\cdots r$.
Letting $D_j \colonequals \div(f_j)$ for $j=1, \cdots, r$ and  $\nu \in M_k$, 
we define    
\[
\il_\Y( \cdot , \nu): (\X \backslash  \Y) (k) \rightarrow \R,
\]
the \textbf{local weighted height  associated to $\Y$}, by 
\begin{equation}\label{gwh1}
	\il_\Y(\x, \nu) \colonequals \min_{1\leq j \leq r}^{}\{ \il_{\widehat{D_j}}(\x, \nu) \} \\
	=\min_{1\leq j \leq r}^{} \left\lbrace   - \log \frac{|f_j(\x)|_\nu}{ \max_{i}^{} \left| x_i\right|_\nu^{\frac{ d_j}{q_i}}}\right\rbrace.
\end{equation}
By convention, we define $\il_\Y(\x, \nu)=\infty$  for $\x \in \Y (k)$.
One can show that this is unique up to a weighted $M_k$-bounded function by a similar argument for the projective varieties.

Recall that  for  closed subschemes  $\Y_1$ and $\Y_2$   of $\X$ defined over $k$ with  corresponding  ideal sheaves $\cI_{\Y_1}$, $\cI_{\Y_2}$,  the closed subschemes 
$\Y_1 \cap \Y_2$,  $\Y_1 + \Y_2$, and $\Y_1 \cup \Y_2$ are defined by ideal sheaves
$\cI_{\Y_1} + \cI_{\Y_2}$,  $\cI_{\Y_1}  \cI_{\Y_2}$, and $\cI_{\Y_1} \cap  \cI_{\Y_2}$ respectively.
Note that  $\Y_1 \cup \Y_2 \subset \Y_1 + \Y_2 \subset \X $ as schemes, since $\cI_{\Y_1}  \cI_{\Y_2} \subset\cI_{\Y_1} \cap  \cI_{\Y_2} $

The basic properties of weighted local heights associated to closed subschemes  are  given in the following proposition.

\begin{prop}\label{lwhcs}
	For any $\nu \in M_k$,  and a closed subscheme $\Y$ of a weighted projective variety $\X $,  the following hold:
	
	\begin{enumerate}[\upshape(1), nolistsep]
		\item  $\il_{\Y_1 \cap \Y_2}(\cdot, \nu)=\min \{ \il_{\Y_1}(\cdot, \nu),\il_{ \Y_2}(\cdot, \nu)\}$;
		\item   $\il_{\Y_1 + \Y_2}(\cdot, \nu)=  \il_{\Y_1}(\cdot, \nu) +\il_{ \Y_2}(\cdot, \nu);$
		\item  $\il_{\Y_1 }(\cdot, \nu) \leq \il_{  \Y_2}(\cdot, \nu) $ if $\Y_1 \subset \Y_2$;  
		\item$\max \{ \il_{\Y_1}(\cdot, \nu),\il_{ \Y_2}(\cdot, \nu)\} \leq \il_{\Y_1 \cup \Y_2}(\cdot, \nu) \leq  \il_{\Y_1}(\cdot, \nu) +\il_{\Y_2}(\cdot, \nu);$
		\item   $\il_{\Y_1 }(\cdot, \nu) \leq  c \cdot \il_{ \Y_2}(\cdot, \nu) $ if $\Supp(\Y_1) \subset \Supp(\Y_2)$ for some constant $c>0$, where $\Supp(\Y)$ denotes the support of $\Y$;
		
		\item If $\Y= D$ is an effective divisor, then $\il_{\Y }(\cdot, \nu)$ is equal  $\il_{\widehat{D}}(\cdot, \nu)$  defined by \cref{wlwh1}, where  $\widehat{D}=\left( \cO_\X(D),   \|\cdot \| \right)  \in \widehat{\Pic_\w(\X)}$;
		
		\item If $\phi: \X' \rightarrow \X $ is a morphism of weighted projective varieties, $\Y \subset \X$ a closed subscheme  over $k$,  and  $\phi^{*}(\Y )$ denotes the closed subscheme of $\X'$ associated to ideal sheaf  $\phi^{-1} \cI_{\Y} \cdot \cO_{\X'}$,   then
		$\il_{\phi^*(\Y)}(\x, \nu)=\il_{\Y}(\phi(\x),\nu )$    for $  \x\in  \left( \X' \backslash  \phi^{*}(\Y )\right) (k).$
	\end{enumerate}
\end{prop}

The \textbf{global weighted height associated  to $\Y$}, can be defined  up to a bounded function by summing all local weighted heights. More precisely, given $\x \in \X$, we let $K$ be a finite extension of $k$ containing   $k(\x)$ and define:
\begin{equation}\label{gwh2}
	\hn_{\Y}(\x)\colonequals\sum_{u\in M_K}^{} \il_{ \Y} (\x, u),
\end{equation}
which  is independent of the choice of the field $K$.   The weighted global heights satisfy similar properties,  except the first one, as given in   \cref{lwhcs} for the weighted local heights.
\begin{prop}\label{wghcs}
	For any $\nu \in M_k$,  and a closed subscheme $\Y$ of a weighted projective variety $\X $  the following hold:
	\begin{enumerate}[\upshape(1), nolistsep]
		\item  $\hn_{\Y_1 \cap \Y_2} \leq \min \{ \hn_{\Y_1}, \hn_{ \Y_2}\}$;
		\item   $\hn_{\Y_1 + \Y_2}=  \hn_{\Y_1} +\hn_{ \Y_2};$
		\item  $\hn_{\Y_1 } \leq \hn_{\Y_2}$ if $\Y_1 \subset \Y_2$;  
		\item$\max \{ \hn_{\Y_1}, \hn_{ \Y_2}\} \leq \hn_{\Y_1 \cup \Y_2} \leq  \hn_{\Y_1} +\hn_{ \Y_2};$
		\item   $\hn_{\Y_1} \leq  c\cdot  \hn_{  \Y_2}$ if $\Supp(\Y_1) \subset \Supp(\Y_2)$ for some constant $c>0$;
		\item If $\Y= D$ is an effective divisor, then $\hn_{\Y }$ is equal to  $\hn_{\widehat{D}}$defined by  \cref{gwh0}, where $\widehat{D}=\left( \cO_\X(D),  ( \|\cdot \|_u ) \right)  \in \widehat{\Pic_\w(\X)}$;
		\item If $\phi: \X' \rightarrow \X $ is a morphism of weighted projective varieties, $\Y \subset \X$ a closed subscheme  over $k$,  then $\hn_{\phi^*(\Y)}=\hn_{\Y}\circ \phi$.
	\end{enumerate}
\end{prop}

All of the  above assertions  follow by summing from the corresponding properties for the local weighted heights associated to subschemes.
When we want to  emphasize on  the base weighted variety  $\X$ in any of the previously defined global weighted heights, we will put it as a subscript on them for example $\hn_{\X, D}$ and  $\hn_{\X, \Y}$.

\section{Conclusion}
This work is devoted to  develop  a detailed theory of Cartier divisors, analytic structure of weighted varieties, weighted blow-ups.  While it was believed that  these results could be recovered from the Veronese embedding it is the first time that  a direct approach is presented. 

Weighted projective spaces  are   very natural objects which makes the theory of weighted heights a powerful tool of arithmetic geometry. However, connections of weighted heights with other heights such as GIT height, Neron-Tate height, Faltings height, etc are not well understood. Some glimpses of the connection between weighted heights and GIT height can be seen in \cite{curri}, but  overall this is an area that offers many open questions. 
Vojta's conjecture for weighted varieties in  terms of weighted heights is studied in \cite{vojta-23}. 

%
\bibliographystyle{abbrv} 
\bibliography{mybib5}{}

\end{document}